\documentclass[11pt,
,a4paper]{article} 
\pdfoutput=1 
\usepackage{amsmath}
\usepackage{amssymb}
\usepackage{amsthm}
\usepackage{verbatim}
\usepackage{enumerate}
\usepackage{color}
\usepackage{blkarray}
\usepackage{bm}

\usepackage{tikz}
\usetikzlibrary{cd}
\usepackage{graphicx,amscd,mathrsfs}

\usepackage[top=30truemm,bottom=30truemm,left=30truemm,right=30truemm]{geometry}

\usepackage[affil-it]{authblk}

\theoremstyle{definition}
\newtheorem{definition}{Definition}[section]
\newtheorem{remark}[definition]{Remark}

\theoremstyle{plain}
\newtheorem{theorem}[definition]{Theorem}
\newtheorem{lemma}[definition]{Lemma}
\newtheorem{proposition}[definition]{Proposition}
\newtheorem{corollary}[definition]{Corollary}

\usepackage{mathtools}
\mathtoolsset{showonlyrefs,showmanualtags}

\usepackage{comment}
\usepackage{color}

\theoremstyle{remark}

\numberwithin{equation}{section}

\def\ds{\displaystyle}

\DeclareMathOperator{\im}{Im}
\DeclareMathOperator{\Ker}{Ker}

\def\bbC{\mathbb{C}}

\def\calL{\mathcal{L}}
\def\calK{\mathcal{K}}

\title{$q$-middle Laplace transform and its applications}
\date{}

\author[1]{Shunya Adachi} 
\author[2]{Takahiko Nobukawa}
\affil[1]{Cooperative Faculty of Education, Utsunomiya University\linebreak
  350 Mine-machi, Utsunomiya-shi, Tochigi, 321-8505 JAPAN\linebreak
 email: {\tt sadachi@a.utsunomiya-u.ac.jp}}
\affil[2]{Department of Mathematical Sciences, Aoyama Gakuin University\linebreak
  5-10-1 Fuchinobe, Chuo-ku, Sagamihara-shi, Kanagawa, 252-0206 JAPAN\linebreak
 email: {\tt nobukawa@math.aoyama.ac.jp}}

\usepackage[pdftex]{hyperref} 
\hypersetup{
	bookmarksnumbered=true,
	colorlinks=true,
	linkcolor=blue, 
	citecolor=blue,
	urlcolor=blue
}
\begin{document}
\maketitle
\allowdisplaybreaks

\begin{abstract}

We introduce a $q$-analog of the one-variable middle Laplace
transform defined by the first author.
We establish its invertibility and preservation of irreducibility,
and derive a decomposition of $q$-middle convolution
in terms of this transform.

As applications, we obtain integral representations of solutions to the confluent $q$-hypergeometric equation and its variant, and construct a transformation of the $2\times 2$ Lax linear equation associated with the $q$-Painlev\'{e} equation of type $A_4^{(1)}$.

\end{abstract}

{\renewcommand{\thefootnote}{}%
	\addtocounter{footnote}{-1}}

\section{Introduction}\label{sec:intro}

Since the groundbreaking work of N.~M.~Katz \cite{KatzBook} on rigid local systems,
the study of differential equations in the complex domain
and special functions has advanced considerably.
A central notion of Katz's theory is middle convolution. 
This was originally introduced as an operation on local systems, and Dettweiler and Reiter \cite{DR2007} later formulated it for Fuchsian systems of linear ordinary differential equations:
    \begin{equation}
        \frac{dY}{dx}=\left(\sum_{i=1}^{m}\frac{A_i}{x-t_i}\right)Y, \quad 
        (A_{i}\in\mathrm{Mat}(n,\bbC)).
    \end{equation}
In particular, this operation is closely related to the Euler-type integral transform
\begin{equation}
    f(x)~\mapsto~
    \int_{\Delta}f(t)(x-t)^{\lambda}\,dt
    \quad(\lambda\in\bbC)
\end{equation}
of solutions. 
Middle convolution has several important properties and applications,
including the description of monodromy representations,
the study of (ir)reducibility conditions,
the Deligne--Simpson problem,
the classification of Fuchsian differential equations,
and isomonodromic deformations
(see, e.g., \cite{HaraokaFilipuk2007,Oshima2012,HaraokaBook}).

As a $q$-analog of middle convolution,
H.~Sakai and M.~Yamaguchi \cite{SakaiYamaguchi}
introduced \emph{$q$-middle convolution}
for systems of linear $q$-difference equations
with rational coefficients:
    \begin{equation}
        Y(qx)=\left(B_{\infty}+\sum_{i=1}^{m}\frac{B_{i}}{1-\frac{x}{t_i}}\right)Y(x),
        \quad 
        (B_{\infty},B_i\in\mathrm{Mat}(n,\bbC)),
    \end{equation}
where $q\in\bbC$ and $0<|q|<1$.
Their $q$-middle convolution is closely related to the Euler-type Jackson integral transform
    \begin{equation}\label{eq:SYJackson}
        f(x)~\mapsto~
        \int_{0}^{\infty} f(s)\frac{(q^{\lambda+1}s/x)_{\infty}}{(qs/x)_{\infty}}\,d_q s
        \quad (\lambda\in\bbC)
    \end{equation}
of solutions, where $(a)_{\infty}=\prod_{i=0}^{\infty}(1-aq^i)$.
They also introduced the notion of Fuchsian type
and defined spectral types and indices of rigidity
for $q$-difference systems, and showed that $q$-middle convolution preserves Fuchsian type and, under suitable conditions, irreducibility and the index of rigidity.

As with middle convolution for differential equations, $q$-middle convolution has various applications
to the study of $q$-difference equations and $q$-special functions. 
Sasaki, Takagi and Takemura \cite{SasakiTakagiTakemura2022} applied the $q$-middle convolution to Jimbo--Sakai's Lax linear equation \cite{JimboSakai1996} for the $q$-Painlev\'e equation of type $D_5^{(1)}$.
More precisely, they constructed a transform preserving the form of the Lax equation, and investigated the $q$-middle convolution in terms of affine Weyl group symmetry.
Their work can be viewed as a $q$-analog of Filipuk's result \cite{Filipuk2006} for the Painlev\'e VI equation.
Subsequently, Arai and Takemura \cite{AraiTakemura2023}
reformulated the Jackson integral transformations
associated with $q$-middle convolution
by introducing an additional parameter
and investigating their convergence.
They applied these transformations to obtain
integral representations of solutions to Heine's $q$-hypergeometric equation and its variants.

In a later paper, Arai and Takemura \cite{AraiTakemura2026} reformulated the $q$-middle convolution itself. 
Their reformulation applies to systems of the form
\begin{equation}\label{eq:qDiff}
D_{x}Y=\left(\sum_{i=1}^{m}\frac{B_{i}}{x-t_i}\right)Y
\quad (B_i\in\mathrm{Mat}(n,\bbC)),
\end{equation}
which can be rewritten in the form \eqref{eq:SYeqn}.
Here $D_{x}$ denotes the $q$-derivative operator defined by
\[
D_{x}f(x)=\frac{f(qx)-f(x)}{x(q-1)}
\]
for a function $f(x)$.
Their formulation is associated with
Jackson integral transformations of the form
   \begin{equation}\label{eq:ATJackson}
        f(x)~\mapsto~
        x^{-\lambda}\int_{0}^{\xi\infty} f(s)\frac{(q^{\lambda+1}s/x)_{\infty}}{(qs/x)_{\infty}}\,d_q s 
        \quad (\lambda\in\bbC,~\xi\in\bbC\setminus\{0\}).
    \end{equation}
One advantage of their formulation is the additivity
of $q$-middle convolution: writing $mc_{\lambda}$
for their $q$-middle convolution, we have
\begin{equation}
    mc_{\lambda}\circ mc_{\mu}
    = mc_{\lambda+\mu}
\end{equation}
under suitable conditions, where equality is understood up to equivalence.
They also applied their construction to
third-order linear $q$-difference equations,
including generalized $q$-hypergeometric equations,
and obtained $q$-integral representations
of their solutions.
Using their formulation, the second author \cite{Nobukawa2026} constructed a rank $3$ $q$-difference system admitting affine Weyl group symmetry of type $E_8^{(1)}$.
A generator of $W(E_8^{(1)})$ is expressed in terms of the $q$-middle convolution.
We note that the $q$-middle convolution in \cite{Nobukawa2026} is constructed via the $q$-Okubo type equation.
A similar procedure is discussed in Section \ref{sec:qML}.

\medskip

Recently, the first author \cite{Adachi2026arXiv} introduced the middle Laplace
transform for linear Pfaffian systems with irregular singularities
and established its fundamental properties including a relation to
middle convolution.
In this paper, we introduce the $q$-middle Laplace transform
as a $q$-analog of the one-variable case of this transform.
We describe its action on solutions in terms of Jackson integral
transforms and establish a decomposition of $q$-middle convolution.
We then apply this construction to $q$-special functions and a Lax linear equation for a $q$-Painlev\'{e} equation.

Returning to the differential setting, we observe that the
middle Laplace transform for the systems considered here
coincides with Harnad duality
(see Remark~\ref{rem:Harnad}).
This duality was introduced by Harnad \cite{Harnad1994} in the study
of isomonodromic deformations and is related to the
Fourier--Laplace transform
(cf. \cite{SanguinettiWoodhouse2004,Yamakawa2016}).
Boalch \cite{Boalch2005,Boalch2009} pointed out the relation between Harnad duality and middle convolution. 
Subsequently, Yamakawa \cite{Yamakawa2011} developed a systematic formulation of
Harnad duality, established its invertibility and preservation of
irreducibility.
He also used this framework to generalize
middle convolution to systems of ordinary differential equations with irregular singularities.
The $q$-middle Laplace
transform can therefore be viewed as a $q$-analog of Harnad duality
for systems whose coefficient matrices have only simple finite
poles and a semisimple constant term.
A detailed study from this viewpoint will be presented
in a subsequent paper.

The organization of this paper is as follows.
In Section \ref{sec:qML}, we introduce the $q$-middle Laplace transform for linear $q$-difference systems and establish its properties (invertibility and preservation of irreducibility).
In Section \ref{sec:qmc}, after reviewing the $q$-middle convolution, we show that the $q$-middle convolution can be decomposed using the $q$-middle Laplace transform.
In the subsequent sections, we consider applying the $q$-middle Laplace transform to $q$-special functions. 
In Section \ref{sec:qHGF}, we give Euler--Laplace type integral solutions for the confluent $q$-hypergeometric equation and its variant \cite{matsunawa2019variants} using the $q$-middle Laplace transform.
In Section \ref{sec:qPainleve}, we construct a transform of Murata's Lax linear system \cite{Murata} for the $q$-Painlev\'e equation of type $A_4^{(1)}$.
In Section \ref{sec:summary}, we summarize this paper and discuss related topics.

We fix some notations. 
Sometimes we express $D_{x}$ as $D_{x}^{+}$. 
For a function $f(x)$, we also define operators $D_{x}^{-}$ and $T_{x}$ as
\begin{equation}
	D_{x}^{-}f(x):=\frac{f(q^{-1}x)-f(x)}{x(q^{-1}-1)},
    \quad 
    T_{x}f(x)=f(qx).
\end{equation}
Throughout this paper, we assume $q\in\bbC$ and $0<|q|<1$.
We also define the $q$-shifted factorial
    \begin{equation}
        (a)_{\infty}:=\prod_{i=0}^{\infty}(1-aq^i), \quad 
        (a)_{n}:=\frac{(a)_\infty}{(aq^{n})_{\infty}}
    \end{equation}
for $a\in\bbC$ and $n\in\mathbb{Z}$.
We fix a branch of $\log q$ and define
    \begin{equation}
        q^{\lambda}:=e^{\lambda \log q} \quad (\lambda\in\bbC).
    \end{equation}

\section{$q$-middle Laplace transform}\label{sec:qML}
In this section, we introduce the $q$-middle Laplace transform for systems of linear $q$-difference equations of the form
	\begin{equation}\label{eq:Eqn}
	D_{x}u=\left(A_{\infty}+\sum_{i=1}^{m}\frac{A_{i}}{x-t_i}\right)u
    \quad 
    (A_\infty,A_i\in\mathrm{Mat}(n,\bbC)),
	\end{equation}
where $t_i\neq t_j$ ($i\neq j$). 
In particular, we assume $A_{\infty}$ is a diagonal matrix 
	\begin{equation}\label{eq:Ainf}
	A_{\infty}=\begin{pmatrix}
	\tau_1 I_{n_1} & & & \\
	& \tau_2 I_{n_2} & & \\
	 & & \ddots & \\
	 & & & \tau_{l} I_{n_l}
	\end{pmatrix}
	\end{equation}
where $\tau_{i}\neq \tau_{j}$ ($i\neq j$) and $n_1+n_2+\dots+n_{l}=n$.
\subsection{Definition of the $q$-middle Laplace transform}

For the equation \eqref{eq:Eqn}, we set 
	\begin{equation}\label{eq:vecU}
	U(x):=\left(\frac{u(x)}{x-t_1},\frac{u(x)}{x-t_2},\ldots,\frac{u(x)}{x-t_m}\right)^T.
	\end{equation}
Then, the vector satisfies
	\begin{equation}\label{eq:extOkubo}
	(qx-S)D_{x}U=\left(A-I+A_{\infty}^{\oplus m}(x-S)\right)U,
	\end{equation}
where 
	\begin{equation}\label{eq:qOkubocoeff}
	\begin{aligned}
	&S=\begin{pmatrix}
	t_1 I_n & & & \\
	& t_2 I_n & & \\
	& & \ddots & \\
	& & & t_{m} I_n 
	\end{pmatrix}, 
	\quad 
	A=\begin{pmatrix}
	A_{1} & A_2 & \cdots & A_{m} \\
	A_{1} & A_2 & \cdots & A_{m} \\
	\vdots & \vdots & & \vdots \\
	A_{1} & A_2 & \cdots & A_{m}
	\end{pmatrix}, \\
	&A_{\infty}^{\oplus m}=
	\bigoplus_{i=1}^{m}A_{\infty}=
	\begin{pmatrix}
	A_{\infty} & & & \\
	& A_{\infty} & & \\
	& & \ddots & \\
	& & & A_{\infty}
	\end{pmatrix}.
	\end{aligned}
	\end{equation}
Note that equation \eqref{eq:extOkubo} is a system with polynomial coefficients. 
We call a system of this form a $q$-Okubo type equation.
For $A_\infty=O$, this agrees with the definition
in \cite{Nobukawa2026}.

For the $q$-Okubo type equation \eqref{eq:extOkubo}, we apply the $q$-Laplace transform $L_{q}$ at the level of operators:
\begin{equation}\label{eq:qLaplaceop}
L_{q}:\begin{cases}
D_{x}~&\mapsto~x, \\
~x~&\mapsto~ -q^{-1}T_{x}^{-1}D_{x}=-q^{-1}D_{x}^{-}.
\end{cases}
\end{equation}
Applying this transform, we transform equation \eqref{eq:extOkubo} into
    \begin{equation}
    (-D_{x}^{-}-S)(xV)=\left(A-I+A_{\infty}^{\oplus m}(-q^{-1}D_{x}^{-}-S)\right)V.
    \end{equation}
Noting the $q$-Leibniz rule
\begin{equation}
D_{x}^{-}(xV)=\frac{x}{q}(D_{x}^{-}V)+D_{x}^{-}(x)V=\frac{x}{q}(D_{x}^{-}V)+V,
\end{equation}
we have
    \begin{equation}\label{eq:qOkuboLap}
    \left(x-A_{\infty}^{\oplus m}\right)D_{x}^{-}V=-q\left(A+(x-A_{\infty}^{\oplus m})S\right)V.
    \end{equation}
By multiplying $(x-A_{\infty}^{\oplus m})^{-1}$ from the left, we obtain
    \begin{equation}\label{eq:qLapeq}
    D_{x}^{-}V=q\left(-S+\sum_{j=1}^{l}\frac{B_j}{x-\tau_{j}}\right)V,
    \end{equation}
where 
	\begin{equation}\label{eq:Bj}
	B_{j}=-\begin{pmatrix}
	E_{n_j}A_{1} & E_{n_j}A_{2} & \cdots & E_{n_j}A_{m}\\
	E_{n_j}A_{1} & E_{n_j}A_{2} & \cdots & E_{n_j}A_{m}\\
	\vdots & \vdots & & \vdots \\
	E_{n_j}A_{1} & E_{n_j}A_{2} & \cdots & E_{n_j}A_{m}
	\end{pmatrix}
	\end{equation}
and 
	\begin{equation}
	E_{n_j}=\mathrm{diag}[O_{n_1},\ldots,I_{n_j},\ldots,O_{n_l}]
	=\begin{pmatrix}
	O_{n_1} & & & &\\
	& \ddots & & &\\
	& & I_{n_j} & &  \\
	& & & \ddots & \\
	& & & & O_{n_l}
	\end{pmatrix}
	\end{equation}
for $1\le j\le l$.
\begin{definition}\label{def:qLaplace}
We call the operation that sends the equation \eqref{eq:Eqn} to the equation \eqref{eq:qLapeq} the \emph{$q$-Laplace transform} and denote it by $\calL_{q}^{+}$ or $\calL_{q}$ simply. 
\end{definition}
We then define the $q$-middle Laplace transform. 
Let us define the subspace
    \begin{equation}\label{eq:K}
        \calK:=\left\{
        \begin{pmatrix}
        v_1 \\
        v_2 \\
        \vdots\\
        v_{m}
        \end{pmatrix} 
        \in (\bbC^n)^m
        ~
        \middle|~ 
        v_i \in \Ker A_i ~~(1\le i \le m)\right\}.
    \end{equation}
We can easily verify that $\calK$ is invariant for the action of $(S,B_1,\ldots,B_l)$.
Hence, $(S,B_1,\ldots,B_l)$ induces the linear actions on the quotient space $\bbC^{mn}/\calK$.
We denote the actions by $(\bar{S},\bar{B}_1,\ldots,\bar{B}_l)$ and consider the equation 
    \begin{equation}\label{eq:qmLeq}
        D_{x}^{-}v=q\left(-\bar{S}+\sum_{j=1}^{l}\frac{\bar{B}_j}{x-\tau_{j}}\right)v.
    \end{equation}
\begin{definition}\label{def:qmiddleLaplace}
We call the operation that sends the equation \eqref{eq:Eqn} to the equation \eqref{eq:qmLeq} the \emph{$q$-middle Laplace transform} and denote it by $\mathcal{ML}_{q}^{+}$ or $\mathcal{ML}_{q}$ simply. 
\end{definition}
\begin{remark}\label{rem:projection}
    We can obtain matrix representations of $(\bar{S},\bar{B}_{1},\dots,\bar{B}_{l})$ as follows.
    Let $\{v_1,v_2,\dots,v_k\}$ be a basis of $\calK$.
    For this basis, we add some vectors $w_{k+1},w_{k+2},\dots,w_{mn}$ so that $\{v_1,\ldots,v_{k},w_{k+1},\dots,w_{mn}\}$ gives a basis of $\bbC^{mn}$. 
    By setting $P=(w_{k+1},\dots,w_{mn},v_1,\ldots,v_{k})\in\mathrm{GL}(mn,\bbC)$, we have
    \begin{equation}
        SP=P\left(
        \begin{array}{c|c} 
        \bar{S} & O_{mn-k,k} \\
        \hline 
        * & *\end{array}\right), 
        \quad 
       B_jP=P\left(\begin{array}{c|c} 
        \bar{B}_j & O_{mn-k,k} \\
        \hline 
        * & *\end{array}\right) \quad (1\le j\le l).
    \end{equation}
\end{remark}
\begin{remark}\label{rem:qmLzero}
    By convention, we allow the quotient space in the definition of the $q$-middle Laplace transform to be the zero vector space. 
    The corresponding rank-zero system is called the zero system and is not regarded as irreducible.
    The same convention applies to the inverse $q$-middle Laplace transform defined below.
\end{remark}
We shall explain the analytical meaning of the $q$-middle Laplace transform briefly.
To do this, we introduce the Jackson integral of a function $f(s)$ as follows. 
    \begin{align}
    &\int_{[0,\xi\infty]}f(s)\,d_{q}s=\int_{0}^{\xi\infty}f(s)\,d_{q}s:=
        (1-q)\sum_{n=-\infty}^{\infty}f(\xi q^n)\xi q^n, \\
    &\int_{[0,\xi]}f(s)\,d_{q}s=\int_{0}^{\xi}f(s)\,d_{q}s:=
        (1-q)\sum_{n=0}^{\infty}f(\xi q^n)\xi q^n, \\
&\int_{[\xi_1,\xi_2]}f(s)\,d_{q}s=\int_{\xi_1}^{\xi_2}f(s)\,d_{q}s=\int_{[0,\xi_2]}f(s)\,d_{q}s-\int_{[0,\xi_1]}f(s)\,d_{q}s.
    \end{align}
Here $\xi,\xi_1,\xi_2\in\bbC$. 
Then, we see that the $q$-Laplace transform \eqref{eq:qLaplaceop} for operators can be realized by the Jackson integral transform for $U(x)$: 
    \begin{equation}\label{eq:V(x)int}
    V(x)=\int_{C}U(s)(q(1-q)xs)_{\infty}\,d_{q}s
    \end{equation}
where the path of integration $C$ is taken so that  
    \begin{equation}
        \left[U(s)((1-q)xs)_{\infty}\right]_{s\in \partial C}=0.
    \end{equation}
Here we define
		\begin{align}
			&[f(s)]_{s\in \partial [\tau_1,\tau_2]}:=f(\tau_2)-f(\tau_1),\\
			&[f(s)]_{s\in \partial (a_1C_1+\cdots+a_mC_m)}:=a_1[f(t)]_{t\in\partial C_1}+\cdots +a_m[f(t)]_{t\in \partial C_m}
		\end{align}
and understand $\displaystyle f(\tau\infty)=\lim_{n\to-\infty}f(\tau q^n)$.
Then, the integral \eqref{eq:V(x)int} gives a solution to \eqref{eq:qLapeq}.
Lastly, setting
\begin{equation}
    V(x)=P\left(\begin{array}{c}
        v(x) \\ \hline
        *
    \end{array}
    \right)
\end{equation}
using $P\in\mathrm{GL}(mn,\bbC)$ appearing in Remark \ref{rem:projection}, we obtain the solution to the $q$-middle Laplace transformed equation \eqref{eq:qmLeq}.

In the same way, we can also define the \emph{inverse $q$-middle Laplace transform} for $q$-difference equations of the form
    \begin{equation}\label{eq:downEqn}
        D_{x}^{-}u=\left(A_{\infty}+\sum_{i=1}^{m}\frac{A_i}{x-t_i}\right)u.
    \end{equation}
For this equation, we set the $mn$-vector $U(x)$ by \eqref{eq:vecU}. 
Then we have 
    \begin{equation}\label{eq:downOkubo}
        (q^{-1}x-S)D_{x}^{-}U=\left(A-I+A_{\infty}^{\oplus m}(x-S)\right)U
    \end{equation}
where $S$, $A$, and $A_{\infty}^{\oplus m}$ are given in \eqref{eq:qOkubocoeff}.
For this equation, we apply the inverse $q$-Laplace transform $L^{-}_{q}$
    \begin{equation}\label{eq:invqLaplaceop}
    L^{-}_{q}:\begin{cases}
    ~x~&\mapsto~ D_{x}, \\
    D_{x}^{-}~&\mapsto~-qx.
    \end{cases}
    \end{equation}
Then we have
    \begin{equation}
    \left(q^{-1}D_{x}-S\right)(-qxW)=\left(A-I+A_{\infty}^{\oplus m}(D_{x}-S)\right)W.
    \end{equation}
Using the $q$-Leibniz rule
    \begin{equation}
    D_{x}(xW)=qx(D_{x}W)+W,
    \end{equation}
we have 
    \begin{equation}
   \left(qx+A_{\infty}^{\oplus m}\right)D_{x}W=-\left(A-(qx+A_{\infty}^{\oplus m})S\right)W.
    \end{equation}
By multiplying $(qx+A_{\infty}^{\oplus m})^{-1}$ from the left, we obtain
    \begin{equation}\label{eq:invqmLeq}
    D_{x}W=\left(S+\sum_{j=1}^{l}\frac{B_j/q}{x+\tau_{j}/q}\right)W,
    \end{equation}
where $B_j$ is given in \eqref{eq:Bj}. 
\begin{definition}\label{def:qinvLaplace}
We call the operation that sends the equation \eqref{eq:downEqn} to the equation \eqref{eq:invqmLeq} the \emph{inverse $q$-Laplace transform} and denote it by $\calL_{q}^{-}$. 
\end{definition}
As explained above, the matrices $(S,B_{1},\ldots,B_{l})$ induce the linear actions $(\bar{S},\bar{B}_{1},\ldots,\bar{B}_{l})$ on the quotient space $\bbC^{mn}/\calK$, where $\calK$ is given by \eqref{eq:K}. 
By using them, we consider the equation
    \begin{equation}\label{eq:qinvmiddleLaplace}
        D_{x}w=\left(\bar{S}+\sum_{j=1}^{l}\frac{\bar{B}_{j}/q}{x+\tau_j/q}\right)w.
    \end{equation}
\begin{definition}\label{def:qinvmiddleLaplace}
We call the operation that sends the equation \eqref{eq:downEqn} to the equation \eqref{eq:qinvmiddleLaplace} the \emph{inverse $q$-middle Laplace transform} and denote it by $\mathcal{ML}_{q}^{-}$. 
\end{definition}
We can verify that the Jackson integral transform
    \begin{equation}\label{eq:W(x)int}
        W(x)=\int_{C} \frac{U(s)}{((1-q)xs)_{\infty}}\,d_{q}s
    \end{equation}
realizes the inverse $q$-Laplace transform \eqref{eq:invqLaplaceop} for operators. 
Here the path of integration $C$ is taken so that 
    \begin{equation}
        \left[U\left(\frac{s}{q}\right)\frac{1}{((1-q)sx)_{\infty}}\right]_{s\in \partial C}=0.
    \end{equation}
Thus, the analytical meaning of the inverse $q$-middle Laplace transform can be described in the same way as above, using \eqref{eq:W(x)int} instead of \eqref{eq:V(x)int}. 

\medskip

We note that the $q$-middle Laplace transform $\mathcal{ML}_{q}$ sends equations of the form \eqref{eq:Eqn} to equations of the form \eqref{eq:downEqn}, while the inverse $q$-middle Laplace transform $\mathcal{ML}_{q}^{-}$ sends equations of the form \eqref{eq:downEqn} to equations of the form \eqref{eq:Eqn}.
It is natural to expect that the inverse $q$-middle Laplace transform is indeed inverse to the $q$-middle Laplace transform.
In the next subsection, we shall show that this expectation is justified under a suitable condition.

\subsection{Properties of $q$-middle Laplace transform}
In this section, we show some fundamental properties of the $q$-middle Laplace transform. 
Before this, we introduce the notion of irreducibility and an equivalence relation for linear $q$-difference equations.
\begin{definition}
    A linear $q$-difference equation \eqref{eq:Eqn} or \eqref{eq:downEqn} is \emph{irreducible} if there is no common invariant subspace of the matrices $(A_{\infty},A_1,A_2,\dots,A_{m})$ except $\mathbb{C}^n$ and $\{0\}$. 
    An equation of the form \eqref{eq:Eqn} or \eqref{eq:downEqn} is called \emph{exceptional} if $n=1$ and $A_i=O$ for all $i=1,2,\ldots,m$. 
    An equation which is not exceptional is called \emph{non-exceptional}.
\end{definition}
An exceptional equation corresponds to a scalar equation $D_{x}^{\pm} u=\alpha u$ ($\alpha\in\bbC$). 
\begin{definition}
    Two linear $q$-difference equations
    \begin{equation}\label{eq:equivA}
        D_{x}^{\pm}u=\left(A_{\infty}+\sum_{i=1}^{m}\frac{A_{i}}{x-t_i}\right)u\qquad (A_{\infty},\,A_{i}\in\mathrm{Mat}(n,\bbC))
     \end{equation}
     and
     \begin{equation}\label{eq:equivB}
        D_{x}^{\pm}v=\left(B_{\infty}+\sum_{i=1}^{m}\frac{B_{i}}{x-t_i}\right)v
        \qquad (B_{\infty},\,B_{i}\in\mathrm{Mat}(n,\bbC))
    \end{equation}
    are \emph{equivalent} if there exists a constant invertible matrix $P\in\mathrm{GL}(n,\bbC)$ such that the gauge transform
    $u=Pv$ sends the equation \eqref{eq:equivA} to \eqref{eq:equivB}.
    That is,
    \begin{equation}\label{eq:equiv}
        B_{\infty}=P^{-1}A_{\infty}P, \quad B_{i}=P^{-1}A_{i}P
        \qquad
        (1\le i\le m)
    \end{equation}
    holds. 
\end{definition}
We note that the equation which is equivalent to an irreducible equation is also irreducible. 
Then, the following holds. 
\begin{theorem}\label{thm:qMLproperty}
    Suppose that the equation \eqref{eq:Eqn} is irreducible and non-exceptional. 
    Then the following hold.
    \begin{itemize}
        \item The resulting equation of the composition $\mathcal{ML}_{q}^{-}\circ\mathcal{ML}_{q}^{+}$ for \eqref{eq:Eqn} 
    is equivalent to the original equation \eqref{eq:Eqn}. 
    Namely, it holds that
    \begin{equation}
    \mathcal{ML}_q^{-}\circ\mathcal{ML}_q^{+}=\mathrm{id},
    \end{equation}
    where the symbol $``="$ is understood up to equivalence. 

    \item The resulting equation of $\mathcal{ML}_{q}^{+}$ for \eqref{eq:Eqn} is also irreducible and non-exceptional.
    \end{itemize}
Similar statements also hold for the inverse $q$-middle Laplace transform of the irreducible and non-exceptional equation \eqref{eq:downEqn}.
\end{theorem}
\begin{proof}
As one can see below, the transforms $\mathcal{ML}_{q}^{\pm}$
have the same construction as their differential counterparts $\mathcal{ML}^{\pm}$,
up to nonzero rescalings of the coefficient matrices
and pole positions.
Therefore, the assertion follows from the corresponding
result for differential equations (Theorem~\ref{thm:differentialML} below).
\end{proof}

In the following, we recall the properties of $\mathcal{ML}^\pm$ in the formulation of
\cite{Adachi2026arXiv}.
For a linear ordinary differential equation of the form
\begin{equation}\label{eq:LODE}
    \frac{du}{dx}
    =\left(A_{\infty}+\sum_{i=1}^{m}\frac{A_i}{x-t_i}\right)u,
\end{equation}
the middle Laplace transform $\mathcal{ML}^{+}$ gives
    \begin{equation}
        \frac{dv}{dx}=\left(-\bar{S}+\sum_{j=1}^{l}\frac{\bar{B}_{j}}{x-\tau_j}\right)v,
    \end{equation}
where $(\bar{S},\bar{B}_1,\ldots,\bar{B}_{l})$ are the same as those appearing in the $q$-middle Laplace transform~\footnote{In \cite{Adachi2026arXiv}, the author used the notations $\mathcal{ML}$ and $\mathcal{ML}^{-1}$ instead of $\mathcal{ML}^{+}$ and $\mathcal{ML}^{-}$, respectively.},
namely, the linear action of $(S,B_{1},\ldots,B_{l})$ on the quotient space $\bbC^{mn}/\calK$, where $S$, $B_i$, $\calK$ are given by \eqref{eq:qOkubocoeff}, \eqref{eq:Bj} and \eqref{eq:K}, respectively. 
Similarly, applying the inverse middle Laplace transform $\mathcal{ML}^{-}$ to \eqref{eq:LODE}, we have
    \begin{equation}
        \frac{dw}{dx}=\left(\bar{S}+\sum_{j=1}^{l}\frac{\bar{B}_{j}}{x+\tau_j}\right)w.
    \end{equation}
To explain the result for the middle Laplace transform, we introduce the notions of irreducibility and equivalence for linear differential equations, which are defined in the same way as for $q$-difference equations.
\begin{definition}
    A linear differential equation \eqref{eq:LODE} is \emph{irreducible} if there is no common invariant subspace of coefficient matrices $(A_{\infty},A_1,\dots,A_{m})$ except $\mathbb{C}^n$ and $\{0\}$.
    An equation \eqref{eq:LODE} is called \emph{exceptional} if $n=1$ and $A_i=0$ for all $i=1,2,\ldots,m$. 
    An equation which is not exceptional is called \emph{non-exceptional}.
    
    Besides, two linear ordinary differential equations 
    \begin{equation}\label{eq:equivAdiff}
        \frac{du}{dx}=\left(A_{\infty}+\sum_{i=1}^{m}\frac{A_{i}}{x-t_i}\right)u\qquad (A_{\infty},\,A_{i}\in\mathrm{Mat}(n,\bbC))
     \end{equation}
     and
     \begin{equation}\label{eq:equivBdiff}
        \frac{dv}{dx}=\left(B_{\infty}+\sum_{i=1}^{m}\frac{B_{i}}{x-t_i}\right)v
        \qquad (B_{\infty},\,B_{i}\in\mathrm{Mat}(n,\bbC))
    \end{equation}
    are \emph{equivalent} if there exists a constant invertible matrix $P\in\mathrm{GL}(n,\bbC)$ such that the gauge transform
    $u=Pv$ sends the equation \eqref{eq:equivAdiff} to \eqref{eq:equivBdiff}.
    That is, \eqref{eq:equiv} holds.
\end{definition}
Then, the properties of the (inverse) middle Laplace transform are given as follows.
\begin{theorem}[Theorem 2.10 of \cite{Adachi2026arXiv}; cf. Theorem 7 of \cite{Yamakawa2011}]\label{thm:differentialML}
    Suppose that equation \eqref{eq:LODE} is irreducible and non-exceptional. 
    Then the following hold.
    \begin{itemize}
    \item The resulting equations of the compositions $\mathcal{ML}^{-}\circ\mathcal{ML}^{+}$ and $\mathcal{ML}^{+}\circ\mathcal{ML}^{-}$ are both equivalent to the original equation \eqref{eq:LODE}. 
    Namely, it holds that
    \begin{equation}
        \mathcal{ML}^{-}\circ\mathcal{ML}^{+}=\mathrm{id},
        \quad 
        \mathcal{ML}^{+}\circ\mathcal{ML}^{-}=\mathrm{id},
    \end{equation}
    where the symbol $``="$ is understood up to equivalence. 
    \item Both the resulting equations of $\mathcal{ML}^{\pm}$ are also irreducible and non-exceptional.
    \end{itemize}
\end{theorem}
We now give the details of the proof of
Theorem~\ref{thm:qMLproperty}.
The construction of $\mathcal{ML}_{q}^{+}$ involves exactly the same quotient space and induced matrices $(\bar{S},\bar{B}_{1},\ldots,\bar{B}_{l})$ as the middle Laplace transform $\mathcal{ML}^{+}$, except for an overall nonzero factor $q$ in the coefficient matrices.
This factor does not affect irreducibility or exceptionality.
In the subsequent application of $\mathcal{ML}_{q}^{-}$, the factor $q$ is canceled by the division by $q$ in \eqref{eq:invqmLeq}, and the pole positions return to $t_i$. 
Hence, the resulting matrix tuple is the same as that obtained from $\mathcal{ML}^{-}\circ\mathcal{ML}^{+}$. 

\begin{remark}\label{rem:Harnad}
For the ordinary differential system \eqref{eq:LODE}
considered above, the middle Laplace transform
$\mathcal{ML}^{+}$ coincides, up to equivalence, with the
Harnad dual functor as defined in
\cite[Section~2.3]{Yamakawa2016}.
Thus, for systems of the form \eqref{eq:Eqn},
the $q$-middle Laplace transform can be regarded as
a $q$-analog of Harnad duality.
A detailed account of this identification will be given
in a subsequent paper; see Section~\ref{sec:futurework}.
\end{remark}

\section{$q$-middle convolution}\label{sec:qmc}
In this section, we give an interpretation of the $q$-middle convolution in terms of the $q$-middle Laplace transform.

\subsection{A review of $q$-middle convolution}
In 2017, Sakai and Yamaguchi \cite{SakaiYamaguchi} first introduced the $q$-middle convolution for linear $q$-difference equations of the form 
    \begin{equation}\label{eq:SYeqn}
        T_{x}u=\left(A_{\infty}+\sum_{i=1}^{m}\frac{A_{i}}{1-\frac{x}{t_i}}\right)u,
    \end{equation}
where $A_{\infty},A_{i}\in\mathrm{Mat}(n,\bbC)$ (we note that the matrix $A_{\infty}$ is not necessarily diagonal unlike the above discussion). 
Their $q$-middle convolution corresponds to the Jackson integral transformation
    \begin{equation}\label{eq:qEuler}
    f(x) \mapsto \int_{0}^{\infty}f(s)\frac{(q^{\lambda+1} sx^{-1})_{\infty}}{(qsx^{-1})_{\infty}}\,d_q s,
    \end{equation}
where $\lambda\in\bbC$ is a parameter.
Here we remark that the transformation \eqref{eq:qEuler} is a $q$-analog of the Euler transform
    \begin{equation}
    f(x)\mapsto\int_{\Delta}f(s)(1-s/x)^{-\lambda}\,ds.
    \end{equation}
Sakai and Yamaguchi also studied the fundamental properties of the $q$-middle convolution such as composition properties, irreducibility, rigidity, and the relation to the spectral types.

After the pioneering work of Sakai and Yamaguchi, Arai and Takemura \cite{AraiTakemura2026} reformulated the $q$-middle convolution.
They rewrote equation \eqref{eq:SYeqn} as
\begin{equation}\label{eq:ATeq}
        D_{x}u=\frac{1}{1-q}\left(\frac{A_{0}}{x}+\sum_{i=1}^{m}\frac{A_{i}}{x-t_i}\right)u
        =\frac{1}{1-q}\left(\sum_{i=0}^{m}\frac{A_{i}}{x-t_i}\right)u
    \end{equation}
with $t_{0}=0$ and $A_{0}=I_{n}-A_{\infty}-A_{1}-\dots-A_{m}$, and reformulated Sakai and Yamaguchi's $q$-middle convolution as the operation for this type of equation.
Their reformulated $q$-middle convolution corresponds to the transform 
    \begin{equation}
        f(x) \mapsto x^{-\lambda} \int_{0}^{\xi\infty} f(s)\frac{(q^{\lambda+1} sx^{-1})_{\infty}}{(qsx^{-1})_{\infty}}\,d_q s,
    \end{equation}
where $\lambda,\xi\in\bbC$ are parameters. 
Although their definition is similar to the one by Sakai and Yamaguchi, the additivity of $q$-middle convolution becomes simpler than in Sakai and Yamaguchi's original formulation. 
In \cite{AraiTakemura2023}, applications of the $q$-middle convolution to the $q$-hypergeometric equation and its variants are also given.
In \cite{Nobukawa2026}, the second author rewrote Arai and Takemura's $q$-middle convolution as an operation for equations of the form
    \begin{equation}\label{eq:Neqn}
        D_{x}u=\left(\sum_{i=1}^{m}\frac{A_i}{x-t_i}\right)u
        \quad
        (A_i \in \mathrm{Mat}(n,\bbC),~t_i\neq t_j\,(i\neq j))
    \end{equation}
and gave an interpretation of the $q$-middle convolution by use of $q$-Okubo type equations. 
We note that, by absorbing the factor $(1-q)^{-1}$ into the residue matrices and relabeling the indices, we can regard equation \eqref{eq:ATeq} as a special case of equation \eqref{eq:Neqn} in which $t_i=0$ for some $i$. 
Thus, the formulation in \cite{Nobukawa2026} is essentially equivalent to that of Arai and Takemura.

Following \cite{Nobukawa2026}, we recall the definition of the $q$-middle convolution. 
Let $\lambda\in\bbC$ be a parameter. 
For the equation \eqref{eq:Neqn}, we define $A, G_1,G_2,\ldots,G_{m}\in\mathrm{Mat}(mn,\bbC)$ by
    \begin{align}
  A&=\begin{pmatrix}
	A_{1} & A_2 & \cdots & A_{m} \\
	A_{1} & A_2 & \cdots & A_{m} \\
	\vdots & \vdots & & \vdots \\
	A_{1} & A_2 & \cdots & A_{m}
	\end{pmatrix},
    \\
    \label{eq:convmat}
    G_{i}&=\begin{pmatrix}
                 &  &  & O& &    \\
                 &  &  & \vdots& &    \\
                 &  &  & O& &    \\
            q^{\lambda}A_{1} & \cdots & q^{\lambda}A_{i-1} & q^{\lambda}A_{i}+[\lambda]I & q^{\lambda}A_{i+1}& \cdots & q^{\lambda}A_{m}\\
                 &  &  & O& &    \\
                 &  &  & \vdots& &    \\
                 &  &  & O &    \\
        \end{pmatrix}
        ~\text{$i$-th block}
    \end{align}
where $[\lambda]:=\ds\frac{1-q^{\lambda}}{1-q}$.
Then the subspaces 
    \begin{align}
        \calK&=\left\{
        \begin{pmatrix}
        v_1 \\
        v_2 \\
        \vdots\\
        v_{m}
        \end{pmatrix} 
        \in (\bbC^n)^m
        ~
        \middle|~ 
        v_i \in \Ker A_i ~~(1\le i \le m)\right\} \subset\bbC^{mn}, 
        \\
        \calK_{\infty}(\lambda)&:=\Ker(G_1+G_2+\cdots+G_m)=\Ker(q^{\lambda}A+[\lambda]I_{mn})\subset \bbC^{mn} \label{eq:qmcKinf}
    \end{align}
are invariant with respect to the action of $(G_1,G_2,\ldots,G_{m})$. 
We remark that the subspace $\calK$ is the same as that in the case of the $q$-middle Laplace transform.
Let $(\bar{G}_{1},\bar{G}_{2},\dots,\bar{G}_{m})$ be the matrices which represent the linear action of $(G_1,G_2,\ldots,G_{m})$ on the quotient space $\bbC^{mn}/(\calK+\calK_{\infty}(\lambda))$.
\begin{definition}\label{def:qmc}
  We refer to the operation that sends \eqref{eq:Neqn} to the equation
    \begin{equation}\label{eq:qmceqn}
    D_{x}v=\left(\sum_{i=1}^{m}\frac{\bar{G}_{i}}{x-t_i}\right)v
    \end{equation}
 as the \emph{$q$-middle convolution} with a parameter $\lambda$ and denote it by $mc_{\lambda}$.
\end{definition}
As in the $q$-middle Laplace transform (Remark \ref{rem:qmLzero}), by convention, we allow the quotient space in the definition of the $q$-middle convolution to be the zero vector space. 
The corresponding rank-zero system is also called the zero system and is not regarded as irreducible.

\begin{theorem}[Theorem 5.9 of \cite{AraiTakemura2026}]
    Suppose that equation \eqref{eq:Neqn} is irreducible.
    Then the resulting equation of $mc_{\lambda}$ is irreducible or the zero system.
\end{theorem}

\begin{remark}
    We note that our parameter $\lambda$, following the convention of \cite{Nobukawa2026}, corresponds to $-\lambda$ in the convention of \cite{AraiTakemura2026}.
\end{remark}

\begin{remark}
    As in the $q$-middle Laplace transform, the solution to equation \eqref{eq:qmceqn} can be represented by the Jackson integral transformation of the initial equation \eqref{eq:Neqn} (cf. \cite{SakaiYamaguchi, AraiTakemura2026}).
    In \cite{Nobukawa2026}, a detailed explanation of this fact is given by use of a $q$-Okubo type equation. 
\end{remark}

We shall briefly recall some important properties of the $q$-middle convolution referring to Arai and Takemura \cite{AraiTakemura2026}. 

\begin{definition}
    For equation \eqref{eq:Neqn}, we define conditions $(*)$ and $(**)$ as follows. 
    \begin{align}
    &\forall i \in\{1,2,\ldots,m\},\,\forall\tau\in\bbC\,; \quad\bigcap_{j\neq i}\Ker A_{j}\cap\Ker(A_{i}+\tau I_{n})  =\{0\}, \tag{$*$}\label{eq:DRstar1}
    \\
    &\forall i \in\{1,2,\ldots,m\},\,\forall\tau\in\bbC\,; \quad\sum_{j\neq i}\im  A_{j}+\im (A_{i}+\tau I_{n})  =\mathbb{C}^n.
    \tag{$**$}\label{eq:DRstar2}
    \end{align}
\end{definition}
These conditions are additive analogs of conditions $(*)$ and $(**)$ introduced by Dettweiler and Reiter \cite{DR2000}.
We note that the equation \eqref{eq:Neqn} satisfies \eqref{eq:DRstar1} and \eqref{eq:DRstar2} under either of the following assumptions: $n\ge2$ and the equation is irreducible; or $n=1$ and at least two of the $A_i$ are nonzero.

\begin{proposition}[Proposition 5.6 of \cite{AraiTakemura2026}; cf. Proposition 4.11 of \cite{SakaiYamaguchi}]
    The conditions \eqref{eq:DRstar1} and \eqref{eq:DRstar2} are preserved by the $q$-middle convolution.
\end{proposition}
\begin{proposition}[Proposition 5.7 of \cite{AraiTakemura2026}; cf. Proposition 4.4 of \cite{SakaiYamaguchi}]\label{prop:AT0}
    Suppose that equation \eqref{eq:Neqn} satisfies \eqref{eq:DRstar2}. 
Then, the resulting equation of $mc_{0}$ is equivalent to the original equation \eqref{eq:Neqn}. 
Namely, it holds that
    \begin{equation}
        mc_{0}=\mathrm{id},
    \end{equation}
where the symbol $``="$ is understood up to equivalence. 
\end{proposition}
\begin{theorem}[Theorem 5.8 of \cite{AraiTakemura2026}\footnote{The composition of Sakai--Yamaguchi's $q$-middle convolution is studied in \cite[Proposition 4.10 and Proposition 4.12]{SakaiYamaguchi}.
}]\label{thm:AT1}
    Suppose that equation \eqref{eq:Neqn} satisfies \eqref{eq:DRstar1} and \eqref{eq:DRstar2}.
    Then, the resulting equation of $mc_{\lambda_2}\circ mc_{\lambda_1}$ is equivalent to the resulting equation of $mc_{\lambda_{1}+\lambda_{2}}$ for any $\lambda_1,\lambda_{2}\in\bbC$. 
    Namely, it holds that
    \begin{equation}
        mc_{\lambda_2}\circ mc_{\lambda_1}=mc_{\lambda_1+\lambda_2},
    \end{equation}
where the symbol $``="$ is understood up to equivalence. 
\end{theorem}
Combining the above results, we see that, if equation \eqref{eq:Neqn} satisfies \eqref{eq:DRstar1} and \eqref{eq:DRstar2}, the $q$-middle convolution $mc_{\lambda}$ is invertible (up to equivalence) for any $\lambda\in\bbC$. 
Indeed, the inverse transformation is given by $mc_{-\lambda}$:
    \begin{equation}
        mc_{\lambda}\circ mc_{-\lambda}=mc_{-\lambda}\circ mc_{\lambda}=mc_{\lambda-\lambda}=mc_{0}=\mathrm{id.}
    \end{equation}

\subsection{$q$-middle Laplace transform and $q$-middle convolution}

In this subsection, we give an interpretation of the $q$-middle convolution for equation \eqref{eq:Neqn} from the point of view of the $q$-middle Laplace transform. 
As a preparation, we define an operation for equations of the form 
    \begin{equation}\label{eq:Neqn2d}
        D_{x}^{-}u=\left(A_{\infty}+\sum_{i=1}^{m}\frac{A_i}{x-t_i}\right)u
    \end{equation}
with $t_j=0$ for some $j$.
For this equation, we set 
    \begin{equation}
        v(x)=x^{\alpha}u(x)\quad (\alpha\in\bbC).
    \end{equation}
Then it satisfies
    \begin{equation}\label{eq:Eqbaddition2d}
        D_{x}^{-}v=\left(q^{-\alpha}A_{\infty}+\sum_{\substack{i=1 \\i\neq j}}^{m}\frac{q^{-\alpha}A_i}{x-t_i}+\frac{q^{-\alpha}A_{j}-q[-\alpha]I}{x}\right)v.
    \end{equation}
Using this correspondence, we define the following.
\begin{definition}\label{def:addition2}
      We call the operation that sends the equation \eqref{eq:Neqn2d} with $t_j=0$ to the equation \eqref{eq:Eqbaddition2d} the \emph{addition} for $t_j=0$ with a parameter $\alpha$ and denote it by $add_{0}^{\alpha}$.   
\end{definition}
    In the differential setting, the middle convolution can be
    expressed in terms of Harnad duality and addition
    (see \cite[Appendix~A]{Boalch2009} and
\cite[Section~6]{Yamakawa2011}).
    The first author established a corresponding decomposition
    in terms of the middle Laplace transform and addition
    (see \cite[Section~5]{Adachi2026arXiv}).
    The following theorem gives a $q$-analog of this decomposition,
    with an additional rescaling of the independent variable.
\begin{theorem}\label{thm:qmcinterp}
    Suppose that $\lambda\in\bbC$ satisfies $[\lambda]\neq0$.
    For equation \eqref{eq:Neqn}, 
    \begin{equation}\label{eq:mclfctr}
        mc_{\lambda}=\mathcal{ML}_{q}^{-}\circ(x\mapsto q^{-\lambda} x)\circ add_0^{-\lambda}\circ\mathcal{ML}_{q}^{+}
    \end{equation}
    as operations, where $``="$ is understood up to equivalence.
    Here $x\mapsto q^{-\lambda} x$ denotes the transform ($q^{-\lambda}$-shift) of the independent variable.
\end{theorem}
\begin{proof}
    For equation \eqref{eq:Neqn}, we apply the $q$-Laplace transform (Definition \ref{def:qLaplace}).
    Then we have
    \begin{equation}
        D_{x}^{-}V=q\left(-S+\frac{B_1}{x}\right)V,
    \end{equation}
    where 
    \begin{equation}
     S=\begin{pmatrix}
            t_1 I_n & & & \\
	          & t_2 I_n & & \\
	         & & \ddots & \\
	         & & & t_{m} I_n 
        \end{pmatrix},
    \quad 
    B_{1}=-A=-\begin{pmatrix}
	A_{1} & A_2 & \cdots & A_{m} \\
	A_{1} & A_2 & \cdots & A_{m} \\
	\vdots & \vdots & & \vdots \\
	A_{1} & A_2 & \cdots & A_{m}
	\end{pmatrix}.
    \end{equation}
Let $\calK$ be the subset defined by \eqref{eq:K}, that is, 
    \begin{equation}
        \calK=\bigoplus_{i=1}^{m}\Ker A_{i}.
    \end{equation}
Then, as explained above, the matrices $(S,B_1)$ induce the linear actions $(\bar{S},\bar{B}_1)$ on the quotient space $\bbC^{mn}/\calK$.
Here we take the basis of $\bbC^{mn}/\calK\cong\bbC^{r}$ ($r=mn-\dim \calK$) so that the matrix representation of $\bar{S}$ is given by
    \begin{equation}
        \bar{S}=\begin{pmatrix}
            t_1 I_{k_1} & & & \\
            & t_2 I_{k_2} & & \\
            & & \ddots & \\
            & & & t_m I_{k_m}
        \end{pmatrix},
        \quad k_i=n-\dim\Ker A_i ~~(1\le i\le m).
    \end{equation}
Then we obtain the $q$-middle Laplace transform of \eqref{eq:Neqn}
    \begin{equation}
        D_x^{-}v=q \left(-\bar{S}+\frac{\bar{B}_1}{x}\right)v.
    \end{equation}
For this equation, we apply the addition $add_{0}^{-\lambda}$.
Then we have
    \begin{equation}
        D_{x}^{-}v_{\mathrm{add}}=\left(-q^{\lambda+1}\bar{S}+\frac{q^{\lambda+1}\bar{B}_1-q[\lambda]I}{x}\right)v_{\mathrm{add}}.
    \end{equation}
We then apply the transform of the independent variable $x\mapsto q^{-\lambda} x$. 
Accordingly, we have the transform of the operator $D_{x}^{-}\mapsto q^{\lambda}D_{x}^{-}$. 
Indeed, setting 
    \begin{equation}
        w(x)=v_{\mathrm{add}}(q^{-\lambda}x)
    \end{equation}
we have
    \begin{equation}
    \begin{aligned}
    D_{x}^{-}v_{\mathrm{add}}(x)&=\frac{v_{\mathrm{add}}(q^{-1}x)-v_{\mathrm{add}}(x)}{(q^{-1}-1)x} \\
    &\mapsto\frac{v_{\mathrm{add}}(q^{-1}q^{-\lambda}x)-v_{\mathrm{add}}(q^{-\lambda}x)}{(q^{-1}-1)q^{-\lambda}x} \\
    &=q^{\lambda}\frac{w(q^{-1}x)-w(x)}{(q^{-1}-1)x}\\
    &=q^{\lambda}D_{x}^{-}w(x).
    \end{aligned}
        \end{equation}
Therefore, we obtain
    \begin{equation}\label{eq:Dxminusw}
        D_{x}^{-}w=\left(-q\bar{S}+\frac{q^{\lambda+1}\bar{B}_1-q[\lambda]I}{x}\right)w.
    \end{equation}
Applying the inverse $q$-Laplace transform (Definition \ref{def:qinvLaplace}) to this equation, we have 
    \begin{equation}
        D_{x}Z=\left(\sum_{i=1}^{m}\frac{C_i}{x-t_i}\right)Z
    \end{equation}
where 
    \begin{equation}
        C_i=-E_{k_i}(q^{\lambda}\bar{B}_1-[\lambda]I)
    \end{equation}
with $E_{k_i}=\mathrm{diag}[O_{k_1},\dots,I_{k_i},\dots,O_{k_m}]$. 
Considering the projection to the quotient space $\bbC^{r}/\Ker(q^{\lambda+1}\bar{B}_1-q[\lambda]I)$, we obtain the inverse $q$-middle Laplace transform of \eqref{eq:Dxminusw} as follows:
    \begin{equation}\label{eq:eqnz}
        D_{x}z=\left(\sum_{i=1}^{m}\frac{\bar{C}_i}{x-t_i}\right)z
    \end{equation}
where $\bar{C}_i$ denotes the matrix that represents the induced action of $C_i$ on $\bbC^{r}/\Ker(q^{\lambda+1}\bar{B}_1-q[\lambda]I)=\bbC^{r}/\Ker(q^{\lambda}\bar{B}_1-[\lambda]I)$. 

We shall show that equation \eqref{eq:eqnz} is equivalent to equation \eqref{eq:qmceqn}, which is the $q$-middle convolution image of \eqref{eq:Neqn}.
Recall that the subspace $\calK_{\infty}(\lambda)$ in \eqref{eq:qmcKinf} for \eqref{eq:Neqn} is given by
    \begin{equation}
        \calK_{\infty}(\lambda)=\Ker(q^{\lambda}A+[\lambda]I)=\Ker(q^{\lambda+1}B_1-q[\lambda]I)=\Ker(q^{\lambda}B_1-[\lambda]I).
    \end{equation}
Now we express the linear transform $P_{\lambda}:=q^{\lambda}B_1-[\lambda]I$ on $\bbC^{mn}$ and the induced linear transform $\bar{P}_{\lambda}=q^{\lambda}\bar{B}_1-[\lambda]I$ on $\bbC^{mn}/\calK$. 
Then it holds that 
    \begin{equation}\label{eq:kerP}
    \Ker \bar{P}_{\lambda}=P_{\lambda}^{-1}(\calK)/\calK.
    \end{equation}
\begin{lemma}
    If $[\lambda]\neq 0$, it holds that
    \begin{equation}\label{eq:lemPlambda}
    P_{\lambda}^{-1}(\calK)=\calK+\Ker P_{\lambda}.
    \end{equation}
\end{lemma}
\begin{proof}
    Let us show $P_{\lambda}^{-1}(\calK)\subset\calK+\Ker P_{\lambda}$. 
    For $v\in P_{\lambda}^{-1}(\calK)$, we set 
    \begin{equation}
        k=-\frac{1}{[\lambda]}P_{\lambda}(v)\in\calK.
    \end{equation} 
    Since $B_1k=0$, we have
    \begin{equation}
        P_{\lambda}(k)=(q^{\lambda}B_1-[\lambda]I)k=-[\lambda]k=P_{\lambda}(v)
    \end{equation}
    and hence $P_{\lambda}(k-v)=0$. 
    This means $v\in \calK+\Ker P_{\lambda}$, and hence $P_{\lambda}^{-1}(\calK)\subset\calK+\Ker P_{\lambda}$.
    The inclusion $P_{\lambda}^{-1}(\calK)\supset\calK+\Ker P_{\lambda}$ is obvious. 
\end{proof}
Thanks to this lemma, equation \eqref{eq:kerP} turns into
    \begin{equation}
        \Ker(q^{\lambda}\bar{B}_{1}-[\lambda]I)=(\calK+\Ker P_{\lambda})/\calK=(\calK+\calK_{\infty}(\lambda))/\calK.
    \end{equation}
Then we have the isomorphism
    \begin{equation}
        \begin{aligned}
           \bbC^{r}/\Ker(q^{\lambda}\bar{B}_1-[\lambda]I)&\cong   (\bbC^{mn}/\calK)/\Ker(q^{\lambda}\bar{B}_1-[\lambda]I) \\  
           &=  \frac{\bbC^{mn}/\calK}{(\calK+\calK_{\infty}(\lambda))/\calK} \\
           &\cong \bbC^{mn}/(\calK+\calK_{\infty}(\lambda))
        \end{aligned}
    \end{equation}
as vector spaces. 
Chasing this isomorphism, we see that the action $\bar{C}_{i}$ on $\bbC^{r}/\Ker(q^{\lambda+1}\bar{B}_1-q[\lambda]I)$ induced by $C_{i}$ is the same as the action on $\bbC^{mn}/(\calK+\calK_{\infty}(\lambda))$ induced by the matrix
    \begin{equation}
        \begin{aligned}
            \begin{pmatrix}
                O_n & & & & \\
                & \ddots & & & \\
                & & I_n & & \\
                & & & \ddots& \\
                & & & & O_n
            \end{pmatrix}
            &\begin{pmatrix}
                q^{\lambda}A_1+[\lambda]I & q^{\lambda}A_2 & \cdots & q^{\lambda}A_{m} \\
                q^{\lambda}A_1 & q^{\lambda}A_2+[\lambda]I & \cdots & q^{\lambda}A_{m} \\
                \vdots & & \ddots & \vdots \\
                q^{\lambda}A_1 & q^{\lambda}A_2 & \cdots & q^{\lambda}A_{m}+[\lambda]I
            \end{pmatrix} \\
            &=\begin{pmatrix}
                 &   & O& &    \\
                 &    & \vdots& &    \\
                 &   & O& &    \\
            q^{\lambda}A_{1} & \cdots &  q^{\lambda}A_{i}+[\lambda]I  & \cdots & q^{\lambda}A_{m}\\
                 &    & O& &    \\
                 &    & \vdots& &    \\
                 &    & O &    \\
        \end{pmatrix}
        ~\text{$i$-th block}
        \\
        &=G_i.
        \end{aligned}
    \end{equation}
This means that there exists an invertible matrix $P$ such that 
    \begin{equation}
        (\bar{C}_1,\bar{C}_2,\ldots,\bar{C}_{m})=(P^{-1}\bar{G}_1P,P^{-1}\bar{G}_2P,\ldots,P^{-1}\bar{G}_{m}P).
    \end{equation}
This shows equation \eqref{eq:eqnz} is equivalent to equation \eqref{eq:qmceqn}.
\end{proof}
\begin{remark}\label{rem:lambda0}
    If equation \eqref{eq:Neqn} satisfies 
    \begin{equation}\label{eq:capker0}
        \bigcap_{i=1}^{m}\Ker A_{i}=\{0\},
    \end{equation}
    then identity \eqref{eq:lemPlambda} holds even if $[\lambda]=0$.
    Indeed, in this case $P_{\lambda}=B_1$ and $\calK\subset\Ker P_{\lambda}$. 
    Thus the identity \eqref{eq:lemPlambda} reduces to $P_{\lambda}^{-1}(\calK)=\Ker P_{\lambda}$, which follows directly from \eqref{eq:capker0}. 
    Therefore, the conclusion of Theorem \ref{thm:qmcinterp} holds for any $\lambda\in\bbC$. 
    In particular, this is the case if equation \eqref{eq:Neqn} satisfies condition \eqref{eq:DRstar1}, or if it is irreducible and non-exceptional.
\end{remark}
We now derive counterparts of Proposition \ref{prop:AT0} and Theorem \ref{thm:AT1} in the irreducible and non-exceptional case from the viewpoint of the $q$-middle Laplace transform.
    \begin{corollary}
    Suppose that equation \eqref{eq:Neqn} is irreducible and non-exceptional. 
    Then
    \begin{equation}\label{eq:mc0Lap}
        mc_{0}=\mathrm{id}.
    \end{equation} 
    Furthermore, for $\lambda \in \mathbb{C}$ such that the resulting equation of $add_0^{-\lambda}\circ\mathcal{ML}_{q}^{+}$ is non-exceptional, the resulting equation of $mc_{\lambda}$ is also irreducible and non-exceptional. 
        In this case,
            \begin{equation}\label{eq:qmcadditive}
            mc_{\mu}\circ mc_{\lambda}=mc_{\lambda+\mu}
            \end{equation}
        holds for any $\mu\in\bbC$.
    Here the symbol $``="$ is understood up to equivalence. 
\end{corollary}
\begin{proof}
    In the following, all identities between operations are understood up to equivalence. 
    Suppose that equation \eqref{eq:Neqn} is irreducible and non-exceptional. 
    Then, by Theorem \ref{thm:qmcinterp} and Remark \ref{rem:lambda0}, we have
    \begin{equation}
        mc_0=\mathcal{ML}_q^{-}\circ\mathcal{ML}_{q}^{+}.
    \end{equation}
    Hence the statement \eqref{eq:mc0Lap} follows from the inverse formula of $\mathcal{ML}_{q}$ (Theorem \ref{thm:qMLproperty}). 

    We next show the remaining statements. 
    We refer to the irreducible and non-exceptional equation \eqref{eq:Neqn} as $(E)$ only in this proof.
    By Theorem \ref{thm:qMLproperty}, the $q$-middle Laplace transformed equation $\mathcal{ML}_{q}^{+}(E)$ is also irreducible and non-exceptional. 
    Take $\lambda\in\mathbb{C}$ such that $add_{0}^{-\lambda}\circ \mathcal{ML}_{q}^{+}(E)$ is non-exceptional. 
    Clearly, the addition preserves irreducibility of equations. 
    The $q^{-\lambda}$-shift $x\mapsto q^{-\lambda}x$ also keeps irreducibility and non-exceptionality. 
    Therefore the equation $(x\mapsto q^{-\lambda}x)\circ add_{0}^{-\lambda}\circ \mathcal{ML}_{q}^{+}(E)$ also fulfills the assumption of Theorem \ref{thm:qMLproperty}. 
    Hence the equation $mc_{\lambda}(E)$ is also irreducible and non-exceptional. 
    Finally, for any $\mu\in\bbC$, we have
    \begin{equation}
    \begin{aligned}
        mc_{\mu}\circ mc_{\lambda}&=\mathcal{ML}_{q}^{-}\circ(x\mapsto q^{-\mu}x)\circ add_{0}^{-\mu}\circ \mathcal{ML}_{q}^{+}\circ\mathcal{ML}_{q}^{-}\circ(x\mapsto q^{-\lambda}x)\circ add_{0}^{-\lambda}\circ\mathcal{ML}_{q}^{+} \\
        &=\mathcal{ML}_{q}^{-}\circ(x\mapsto q^{-\mu}x)\circ add_{0}^{-\mu}\circ (x\mapsto q^{-\lambda}x)\circ add_{0}^{-\lambda}\circ\mathcal{ML}_{q}^{+}.
    \end{aligned}
    \end{equation}
Since
    \begin{equation}
        (x\mapsto q^{-\mu}x)\circ add_{0}^{-\mu}\circ (x\mapsto q^{-\lambda}x)\circ add_{0}^{-\lambda}=(x\mapsto q^{-(\lambda+\mu)}x)\circ add_{0}^{-(\lambda+\mu)},
    \end{equation}
we obtain
    \begin{equation}
        mc_{\mu}\circ mc_{\lambda}=\mathcal{ML}_{q}^{-}\circ (x\mapsto q^{-(\lambda+\mu)}x)\circ add_{0}^{-(\lambda+\mu)}\circ\mathcal{ML}_{q}^{+}=mc_{\lambda+\mu}.
    \end{equation}
\end{proof}

\section{Applications to $q$-hypergeometric equations}\label{sec:qHGF}

In this section, we apply the $q$-middle Laplace transform to the confluent $q$-hypergeometric equation and its variant \cite{matsunawa2019variants}.
\subsection{$q$-Kummer equation}
We consider integral solutions for the $q$-Kummer equation, i.e. the confluent $q$-hypergeometric equation 
\begin{equation}\label{eq:qKummer}
    [(1-T_x)(1-cq^{-1}T_x)+xT_x(1-aT_x)]u=0,
\end{equation}
satisfied by the confluent $q$-hypergeometric series
\begin{align}
    {}_1\varphi_1\left(\begin{array}{c}
         a \\
         c
    \end{array};x\right):=\sum_{n=0}^\infty \frac{(a)_n}{(c)_n(q)_n}(-1)^n q^{n(n-1)/2}x^n.
\end{align}
It is well-known that the $q$-Kummer equation \eqref{eq:qKummer} has an integral solution in the following form:
\begin{align}
    \int_C s^{\alpha-1}\frac{(qs)_\infty}{(cs/a)_\infty}(xs)_\infty \,d_qs,
\end{align}
where $a=q^\alpha$ and $C$ is a suitable path.
We reconstruct this integral solution using the $q$-(middle) Laplace transform in this subsection.

Let $e=q^\rho$ and
\begin{equation}
    u=u(x)=x^\rho \frac{(d_1x)_\infty}{(d_2x)_\infty}.
\end{equation}
From direct calculation, we have
\begin{align}
    D_xu= \left(\frac{A_1}{x}+\frac{A_2}{x-d_1^{-1}}\right)u
\end{align}
where
\begin{align}
    A_1=\frac{1-e}{1-q},\quad A_2=\frac{e(1-d_2/d_1)}{1-q}.
\end{align}
We note that this equation is in the form of \eqref{eq:Eqn} with $A_\infty=0$.
We put
\begin{align}
    U=\left(\frac{u}{x},\frac{u}{x-d_1^{-1}}\right)^{T}.
\end{align}
Applying $L_q$, we get
\begin{align}
    \label{eq:qKummerLaplace}
    D_x^-V=q\left(-S+\frac{B_1}{x}\right)V,
\end{align}
where
\begin{align}
    V=\int_C U(s)(q(1-q)xs)_\infty\,d_qs,\ S=\begin{pmatrix}
        0&\\&d_1^{-1}
    \end{pmatrix},\ B_1=-\begin{pmatrix}
        A_1&A_2\\A_1&A_2
    \end{pmatrix}.
\end{align}
The path $C$ should be taken so that
\begin{equation}
    \left[U(s)((1-q)xs)_{\infty}\right]_{s\in\partial C}=0.
\end{equation}
For instance, this condition holds for $C=[0,1/(d_1q)]$ when $|q^{\rho-1}|<1$ and $d_2 \neq d_1 q^{1-k}$ ($k=0,1,2,\ldots$).
Calculating the equation satisfied by the first component of $V$ from \eqref{eq:qKummerLaplace}, we have
\begin{align}
    e(d_2-(1-q)x)V_1+(-d_1-d_2e+(1-q)x)T_x^{-1}V_1+d_1T_x^{-2}V_1=0.
\end{align}
Applying $T_x^2$ to both sides, and putting $z=q(1-q)x$, $d_1=q$, $d_2=c/a$, $e=a$, we finally get
\begin{align}
    [(1-T_z)(1-cq^{-1}T_z)+zT_z(1-aT_z)]V_1=0.
\end{align}
In conclusion, we have
\begin{proposition}\label{prop:qhyp1}
    We put
    \begin{align}
        u(x)=x^{\alpha}\frac{(qx)_\infty}{(cx/a)_\infty},
    \end{align}
    where $q^\alpha=a$.
    We suppose
    \begin{equation}
        \left[\left(\frac{u(s)}{s},\frac{u(s)}{s-q^{-1}}\right)^{T}(xs/q)_{\infty}\right]_{s\in\partial C}=0.
    \end{equation}
    Then we have
    \begin{equation}
        [(1-T_x)(1-cq^{-1}T_x)+xT_x(1-aT_x)]\int_C s^{\alpha-1}\frac{(qs)_\infty}{(cs/a)_\infty}(xs)_\infty \,d_qs=0.
    \end{equation}
\end{proposition}
\begin{remark}
    The integral
    \begin{align}
        \int_C s^{\alpha-1}\frac{(qs)_\infty}{(cs/a)_\infty}(xs)_\infty \,d_qs
    \end{align}
    is a $q$-analog of
    \begin{align}
        \int_Cs^{\alpha-1}(1-s)^{\gamma-\alpha-1}e^{xs}\,ds
    \end{align}
    with $q^\gamma=c$.
\end{remark}

We can also construct integral solutions for \eqref{eq:qKummer} by using $\mathcal{L}_q^-$.
We recall $u$:
\begin{equation}
    u=u(x)=x^\rho \frac{(d_1x)_\infty}{(d_2x)_\infty},\ q^\rho=e.
\end{equation}
We have
\begin{align}
    D_x^-u=\left(\frac{A_1}{x}+\frac{A_2}{x-q/d_2}\right)u
\end{align}
where
\begin{align}
    A_1=\frac{1-e^{-1}}{1-q^{-1}},\ A_2=\frac{1-d_1/d_2}{(1-q^{-1})e}.
\end{align}
Putting $U=\left(\dfrac{u}{x},\dfrac{u}{x-q/d_2}\right)^T$ and applying $L_q^-$, we get
\begin{align}
    \label{eq:qKummerinverseLaplace}
    D_xW=\left(S+\frac{B_1/q}{x}\right)W,
\end{align}
where 
\begin{align}
    W=\int_C U(s)\frac{1}{((1-q)xs)_\infty}\,d_qs,
    \ 
    S=\begin{pmatrix}
        0&\\&q/d_2
    \end{pmatrix},\ 
    B_1=-\begin{pmatrix}
        A_1&A_2\\A_1&A_2
    \end{pmatrix}.
\end{align}
The condition
\begin{equation}
    \left[\frac{U(s/q)}{((1-q)xs)_\infty}\right]_{s\in\partial C}=0
\end{equation}
on the path $C$ holds for $C=[0,q/d_1]$ when $|q^{\rho-1}|<1$ and $d_2 \neq d_1 q^{1-k}$ ($k=0,1,2,\ldots$), and for $C=[0,\xi\infty]$ when $|q^{\rho-1}|<1$ and $\xi\neq q^n/d_2$ ($n\in\mathbb{Z}$).
The second case can be verified by
\begin{align}
    (X q^{-k})_\infty=(X)_\infty(q/X)_k \left(-\frac{X}{q}\right)^kq^{-\binom{k}{2}}\ (k=0,1,2,\ldots).
\end{align}
Calculating the equation satisfied by the first component of $W$ from \eqref{eq:qKummerinverseLaplace}, we have
\begin{align}
    &[c^2T_z^2-c(c+q-qz)T_z+q(c-az)]W_1=0,\\
    &z=(1-q)x,\ d_1=\frac{cq}{a},\ d_2=c,\ e=\frac{c}{a}.
\end{align}
Applying the gauge transform $W_1'=(az/c)_\infty W_1$, we transform this equation into the $q$-Kummer equation.
Therefore, we have
\begin{proposition}\label{prop:qhyp2}
    We put
    \begin{align}
        u(x)=x^{\gamma-\alpha}\frac{(cqx/a)_\infty}{(cx)_\infty},
    \end{align}
    where $q^\alpha=a$, $q^\gamma=c$.
    We suppose
    \begin{equation}
        \left[\left(\frac{u(s/q)}{s/q},\frac{u(s/q)}{s/q-q/c}\right)^{T}\frac{1}{(xs)_\infty}\right]_{s\in\partial C}=0.
    \end{equation}
    Then we have
    \begin{equation}
        [(1-T_x)(1-cq^{-1}T_x)+xT_x(1-aT_x)](ax/c)_\infty\int_C s^{\gamma-\alpha-1}\frac{(cqs/a)_\infty}{(cs)_\infty}\frac{1}{(xs)_\infty} \,d_qs=0.
    \end{equation}
\end{proposition}
\begin{remark}
    The solution
    \begin{align}
        (ax/c)_\infty\int_C s^{\gamma-\alpha-1}\frac{(cqs/a)_\infty}{(cs)_\infty}\frac{1}{(xs)_\infty} \,d_qs
    \end{align}
    is a $q$-analog of
    \begin{align}
        e^{x}\int_C s^{\gamma-\alpha-1}(1-s)^{\alpha-1}e^{-xs}\,ds.
    \end{align}
\end{remark}

\subsection{Variant of singly confluent $q$-hypergeometric equation}
In \cite{matsunawa2019variants}, a variant of the $q$-Kummer equation \eqref{eq:qKummer} is introduced\footnote{
    This equation is constructed by calculating a confluent limit of a variant of the $q$-hypergeometric equation of degree two \cite{HMST2022}.
    Variants of the $q$-hypergeometric equation have a background of quantum integrable systems.
}:
\begin{align}
    &q^{ h_1+h_2-l_1-2\lambda +1/2}t_2(q^{l_1- 1/2}t_1-x)g(qx) \\
    & +(x-q^{h_1 + 1/2} t_1)(x-q^{h_2 + 1/2}t_2)g(x/q) \\
    &-[q^{\alpha _1}x^2 -q^{h_1+h_2-\lambda +1/2}(q^{-h_2}t_1+q^{-h_1}t_2+q^{-l_1}t_2)x 
     +q^{h_1+h_2-\lambda }(q + 1) t_1t_2]g(x)\\
    &=0.\label{eq:variant}
\end{align} 
We shall construct integral solutions for this equation.
First, we characterize the equation \eqref{eq:variant}.
\begin{proposition}\label{prop:characterizevariant}
    Let $H$ be a $q$-difference operator such that
    \begin{align}
        H=\sum_{i=0}^2\sum_{j=-1}^1 a_{i,j}x^iT_x^j\qquad(a_{i,j}\in\mathbb{C},\ a_{2,1}=0).
    \end{align}
    We suppose the following conditions:
    \begin{itemize}
        \item[1.] We rewrite $H=\sum_{i=0}^2 x^iL_i(T_x)$.
        Then $L_0(T_x)\propto(T_x-q^\lambda)(T_x-q^{\lambda+1})$, $L_1(T_x)\propto(T_x-q^\lambda)$ and $L_2(T_x)\propto(T_x-q^{-\alpha_1})$ hold.
        \item[2.] We rewrite $H=\sum_{j=-1}^1 P_j(x)T_x^j$.
        Then $P_{-1}(x)\propto(x-q^{h_1+1/2}t_1)(x-q^{h_2+1/2}t_2)$ and $P_1(x)\propto(x-q^{l_1-1/2}t_1)$ hold.
    \end{itemize}
    Then the equation $Hg(x)=0$ coincides with \eqref{eq:variant}.
\end{proposition}
\begin{proof}
    We write the equation \eqref{eq:variant} as $\mathcal{H}g(x)=0$.
    It can be easily checked that the operator $\mathcal{H}$ satisfies the conditions.
    On the other hand, the operator $H$ is determined by the conditions, up to an overall constant as follows.
    The condition 2 determines the ratios $a_{0,-1}/a_{1,-1}$, $a_{0,-1}/a_{2,-1}$, $a_{0,1}/a_{1,1}$.
    The condition 1 for $L_0$ and $L_2$ also determines $a_{0,-1}/a_{0,0}$, $a_{0,-1}/a_{0,1}$, $a_{2,-1}/a_{2,0}$.
    Since $a_{1,-1}/a_{1,1}$ is determined already, the condition 1 for $L_1$ determines $a_{1,-1}/a_{1,0}$.
    This means that each $a_{i,j}/a_{0,1}$ is determined.
\end{proof}

We construct integral solutions for \eqref{eq:variant} in the following.
We first show a construction via $\mathcal{L}_q$.
Put
\begin{align}
    u=u(x)=x^\rho \frac{(d_1x)_\infty}{(d_3x)_\infty(d_4x)_\infty},\ q^\rho=e.
\end{align}
Then we have
\begin{align}
    D_xu=\left(A_\infty+\frac{A_1}{x}+\frac{A_2}{x-d_1^{-1}}\right)u,
\end{align}
where
\begin{align}
   A_\infty=\frac{d_3d_4e}{(1-q)d_1},\quad  
   A_1=\frac{1-e}{1-q},\quad 
   A_2=\frac{e(d_1-d_3)(d_1-d_4)}{d_1^2(1-q)}.
\end{align}
We put
\begin{align}
    U=\left(\frac{u}{x},\frac{u}{x-d_1^{-1}}\right)^{T}.
\end{align}
Applying $L_q$, we get
\begin{align}
    D_x^-V=q\left(-S+\frac{B_1}{x-\tau_1}\right)V,
\end{align}
where $\tau_1=d_3d_4e/((1-q)d_1)$ and
\begin{align}
    V=\int_C U(s)(q(1-q)xs)_\infty \,d_qs,\ S=\begin{pmatrix}
        0&\\&d_1^{-1}
    \end{pmatrix},\ B_1=-\begin{pmatrix}
        A_1&A_2\\A_1&A_2
    \end{pmatrix}.
\end{align}
Here the path $C$ is taken so that $\left[U(s)((1-q)xs)_\infty\right]_{s\in\partial C}=0$.
Calculating the equation satisfied by the first component of $V$, we have
\begin{align}
    &H V_1=0,\\
    &H=(eqd_3d_4-d_1z)T_z^{-1}+(-(1+q)ed_3d_4+(d_1+e(d_3+d_4))z-z^2)+e(z-d_3)(z-d_4)T_z\\
    &\phantom{H}=-z^2(1-eT_z)-zT_z^{-1}(1-T_z)(d_1-e(d_3+d_4)T_z)+ed_3d_4T_z^{-1}(1-T_z)(q-T_z),\\
    &z=(1-q)qx.
\end{align}
We apply a gauge transformation: $y=\dfrac{z^\lambda}{(z/d_3)_\infty}V_1$.
Then we have
\begin{align}
    H' y=0,
\end{align}
where
\begin{align}
    H'&=(d_1z-eqd_3d_4)(z-qd_3)q^\lambda T_z^{-1}+qd_3(-(1+q)ed_3d_4+(d_1+e(d_3+d_4))z-z^2)\\ 
    &\qquad -d_3^2eq(z-d_4)q^{-\lambda}T_z\\
    &=z^2q^\lambda T_z^{-1}(d_1-d_3q^{1-\lambda}T_z)-q d_3z T_z^{-1}(q^\lambda-T_z)(d_1+ed_4-ed_3q^{-\lambda}T_z)\\
    &\qquad+ed_3^2d_4q^{1-\lambda}T_z^{-1}(q^\lambda-T_z)(q^{1+\lambda}-T_z).
\end{align}
Putting
\begin{align}
    d_1=q^{h_2-\alpha_1-\lambda+1/2}t_2,\ d_3=q^{h_2-1/2}t_2,\ d_4=q^{l_1-1/2}t_1,\ e=q^{h_1-l_1-\alpha_1-\lambda+1},
\end{align}
we find that the operator $H'$ satisfies the conditions 1 and 2 in Proposition \ref{prop:characterizevariant}.
Therefore we get an integral solution of \eqref{eq:variant}:
\begin{align}
    \frac{x^\lambda}{(x/(q^{h_2-1/2}t_2))_\infty}\int_C s^{h_1-l_1-\alpha_1-\lambda}\frac{(q^{h_2-\alpha_1-\lambda+1/2}t_2s)_\infty}{(q^{h_2-1/2}t_2s)_\infty(q^{l_1-1/2}t_1s)_\infty}(xs)_\infty \,d_qs.
\end{align}
Using $L_q^{-}$, we also construct an integral solution of \eqref{eq:variant}.
We put
\begin{align}
    u=x^\rho \frac{(d_1x)_\infty(d_2x)_\infty}{(d_3x)_\infty},\qquad q^\rho=e.
\end{align}
Then we get
\begin{align}
    D_x^-u=\left(A_\infty+\frac{A_1}{x}+\frac{A_2}{x-q/d_3}\right)u,
\end{align}
where
\begin{align}
    A_\infty=-\frac{d_1d_2}{(1-q)d_3e},\ A_1=\frac{(1-e)q}{(1-q)e},\ A_2=-\frac{q(d_1-d_3)(d_2-d_3)}{(1-q)d_3^2e}.
\end{align}
Putting 
\begin{align}
    U=\left(\frac{u}{x},\frac{u}{x-q/d_3}\right)^T,\ W=\int_C \frac{U(s)}{((1-q)xs)_\infty}\,d_qs,
\end{align}
where $\left[\dfrac{U(s/q)}{((1-q)xs)_\infty}\right]_{s\in\partial C}=0$, we have
\begin{align}
    &D_xW=\left(S+\frac{B_1/q}{x+\tau_1/q}\right)W,\\
    &S=\begin{pmatrix}
        0&0\\0&q/d_3
    \end{pmatrix},\ B_1=-\begin{pmatrix}
        A_1&A_2\\A_1&A_2
    \end{pmatrix},\ \tau_1=-\frac{d_1d_2}{(1-q)d_3e}.
\end{align}
By a simple calculation, we have
\begin{align}
    &HW_1=0,\\
    &H=q(z-d_1)(z-d_2)T_z^{-1}+(-eqz^2+qz(d_1+d_2+ed_3)-(1+q)d_1d_2)+(d_1d_2-eqd_3z)T_z\\
    &\phantom{H}=d_1d_2(1-T_z)(q-T_z)T_z^{-1}-qz(1-T_z)(d_1+d_2-ed_3T_z)T_z^{-1}+qz^2(1-eT_z)T_z^{-1},\\
    &z=(1-q)x,
\end{align}
where $W_1$ is the first component of $W$.
Comparing $H$ with the conditions in Proposition \ref{prop:characterizevariant}, we find that the equation satisfied by $x^\lambda W_1$ coincides with \eqref{eq:variant} by putting
\begin{align}
    d_1=q^{h_1+1/2}t_1,\ d_2=q^{h_2+1/2}t_2,\ d_3=q^{h_1+h_2-l_1-\alpha_1-\lambda+1/2}t_2,\ e=q^{\alpha_1+\lambda}.
\end{align}
Therefore we have the following.
\begin{proposition}\label{prop:qvarhyp}
    For suitable paths, the following integrals are solutions of the variant of the confluent $q$-hypergeometric equation \eqref{eq:variant}:
    \begin{align}
     &\frac{x^\lambda}{(x/(q^{h_2-1/2}t_2))_\infty}\int_C s^{h_1-l_1-\alpha_1-\lambda}\frac{(q^{h_2-\alpha_1-\lambda+1/2}t_2s)_\infty}{(q^{h_2-1/2}t_2s)_\infty(q^{l_1-1/2}t_1s)_\infty}(xs)_\infty \,d_qs,\\
     \label{eq:intsolvariant}&x^\lambda\int_C s^{\alpha_1+\lambda-1}\frac{(q^{h_1+1/2}t_1s)_\infty(q^{h_2+1/2}t_2s)_\infty}{(q^{h_1+h_2-l_1-\alpha_1-\lambda+1/2}t_2s)_\infty}\frac{1}{(xs)_\infty}\,d_qs.
\end{align}
\end{proposition}
\begin{remark}
    In \cite{matsunawa2019variants}, several solutions for the equation \eqref{eq:variant} are given in terms of series of hypergeometric type.
    For example, the following series is a solution:
    \begin{align}
        \label{eq:MSTsol}x^\lambda {}_2\varphi_1\left(\begin{array}{c}
            q^{\lambda+\alpha_1} ,   q^{h_i+1/2}t_i/x\\
             q^{h_i-l_1+1}t_i/t_1
        \end{array};\frac{x}{q^{h_j-1/2}t_j}\right),
    \end{align}
    where $(i,j)=(1,2)$ or $(2,1)$ and
    \begin{align}
        {}_2\varphi_1\left(\begin{array}{c}
            a,b\\c
        \end{array};x\right)=\sum_{n=0}^\infty \frac{(a)_n(b)_n}{(c)_n(q)_n}x^n.
    \end{align}
    Using Heine's transformation, we have
    \begin{align}
        &x^\lambda {}_2\varphi_1\left(\begin{array}{c}
            q^{\lambda+\alpha_1} ,   q^{h_i+1/2}t_i/x\\
             q^{h_i-l_1+1}t_i/t_1
        \end{array};\frac{x}{q^{h_j-1/2}t_j}\right)\\
        &=x^\lambda \frac{(q^{\lambda+\alpha_1})_\infty(q^{h_i-h_j+1}t_i/t_j)_\infty}{(q^{h_i-l_1+1}t_i/t_1)_\infty(q^{-h_j+1/2}x/t_j)_\infty}{}_2\varphi_1\left(\begin{array}{c}
            q^{h_i-l_1+1-\lambda-\alpha_1}t_i/t_1,q^{-h_j+1/2}x/t_j\\
            q^{h_i-h_j+1}t_i/t_j
        \end{array};q^{\lambda+\alpha_1}\right)\\
        &=\frac{(q^{\lambda+\alpha_1})_\infty(q^{h_i-l_1+1-\lambda-\alpha_1}t_i/t_1)_\infty}{(1-q)(q^{h_i-l_1+1}t_i/t_1)_\infty(q)_\infty}(q^{h_j-1/2}t_j)^{\lambda+\alpha_1}\\
        &\qquad\times x^\lambda\int_0^{q^{-h_j+1/2}/t_j}s^{\alpha_1+\lambda-1}\frac{(q^{h_1+1/2}t_1s)_\infty(q^{h_2+1/2}t_2s)_\infty}{(q^{h_1+h_2-l_1-\alpha_1-\lambda+1/2}t_2s)_\infty}\frac{1}{(xs)_\infty}\,d_qs.
    \end{align}
    Our solution \eqref{eq:intsolvariant} coincides with \eqref{eq:MSTsol} up to a constant when $C=[0,q^{-h_j+1/2}/t_j]$.
\end{remark}

\section{Application to $q$-Painlev\'e equation of type $A_4^{(1)}$}\label{sec:qPainleve}

In this section, we apply $\mathcal{ML}_q$ to Murata's Lax linear equation \cite{Murata} for the $q$-Painlev\'e equation of type $A_4^{(1)}$:
\begin{align}
    \label{eq:murata}&T_x u= A(x)u,\\
    &A(x)=A(x,t)=A^{(0)}+A^{(1)}x+A^{(2)}x^2,\\
    &\det A(x)=\kappa_1\kappa_2 (x-a_1t)(x-a_2t)(x-a_3),\ 
    A^{(0)}\sim \begin{pmatrix}
        \theta_1 t&0\\
        0&\theta_2 t
    \end{pmatrix}, \ 
    A^{(2)}=\begin{pmatrix}
        \kappa_1&0\\
        0&0
    \end{pmatrix}.
\end{align}
The $q$-Painlev\'e equation of type $A_4^{(1)}$ can be written as the compatibility condition of \eqref{eq:murata} and a suitable deformation $T_t u=B(x,t)u$.
We construct a transform of \eqref{eq:murata} using $\mathcal{ML}_q$.
\begin{theorem}\label{Thm:qPA4}
    We apply the following transform to the equation \eqref{eq:murata} step by step:
    \begin{itemize}
        \item[1.] gauge transform $u\mapsto x^\rho (x/(a_1t))_\infty u$, where $q^\rho=\theta_1^{-1}t^{-1}$
        \item[2.] $q$-middle Laplace transform $\mathcal{ML}_q$
        \item[3.] gauge transform $u\mapsto x^\mu ((1-q)a_2tqx)_{\infty}^{-1}((1-q)a_3qx)_{\infty}^{-1}u$, where $q^\mu=\theta_2 t$
        \item[4.] scaling $z=(1-q)qa_2a_3tx$.
    \end{itemize}
    Then the transformed equation $T_zu=\bar{A}(z)u$ satisfies the following properties:
    \begin{align}
        &\bar{A}(z)=\bar{A}^{(0)}+\bar{A}^{(1)}z+\bar{A}^{(2)}z^2,\\
        &\det \bar{A}(z)=\kappa_1\kappa_2(z-a_1a_2a_3\kappa_1t/\theta_1)(z-a_2t)(z-a_3),
        \\
        &\bar{A}^{(0)}\sim \begin{pmatrix}
            a_2a_3\kappa_1t&\\&\theta_2t
        \end{pmatrix},
        \quad
        \bar{A}^{(2)}=\begin{pmatrix}
            \dfrac{\theta_1}{a_2a_3}&\\&0
        \end{pmatrix}.
    \end{align}
    In particular, a transform
    \begin{align}
      a_1\mapsto -\frac{\theta_2}{\kappa_2},\ \kappa_1\mapsto \frac{\theta_1}{a_2a_3},\ \kappa_2\mapsto-\frac{\theta_2}{a_1},\ \theta_1\mapsto a_2a_3\kappa_1
  \end{align}
  for the equation \eqref{eq:murata} is given by using $\mathcal{ML}_q$.
\end{theorem}
\begin{proof}
    First, we consider the equation satisfied by $u'=x^\rho (x/(a_1t))_\infty u$, where $q^\rho=\theta_1^{-1}t^{-1}$.
    From the equation \eqref{eq:murata}, we have
    \begin{align}
        T_x u'=\frac{1}{\theta_1t(1-x/(a_1t))}A(x)u'.
    \end{align}
    We rewrite this equation by using $D_x$ as follows:
    \begin{align}
       \label{eq:murataDx} D_xu'=\left(A_\infty+\frac{A_1}{x}+\frac{A_2}{x-a_1t}\right)u'.
    \end{align}
    Here
    \begin{align}
        &A_1=\left(\frac{I-\frac{1}{\theta_1t(1-x/(a_1t))}A(x)}{(1-q)x}x\right)\bigg|_{x\to0}=\frac{I-\frac{1}{\theta_1t}A^{(0)}}{1-q},
        \\
        &A_2=\left(\frac{I-\frac{1}{\theta_1t(1-x/(a_1t))}A(x)}{(1-q)x}(x-a_1t)\right)\bigg|_{x\to a_1t}=\frac{\frac{1}{\theta_1t}A(a_1t)}{1-q},
        \\
        &A_\infty=\lim_{x\to \infty}\frac{I-\frac{1}{\theta_1t(1-x/(a_1t))}A(x)}{(1-q)x}=\frac{a_1}{\theta_1(1-q)}A^{(2)}=\begin{pmatrix}
            \frac{a_1\kappa_1}{\theta_1(1-q)}&0\\
            0&0
        \end{pmatrix}.
    \end{align}
    Next, we apply the $q$-middle Laplace transform $\mathcal{ML}_q$ to \eqref{eq:murataDx}.
    We calculate the invariant space 
    \begin{align}
        \mathcal{K}=\left\{\begin{pmatrix}
            v_1\\v_2
        \end{pmatrix}\mid v_j\in\Ker A_j\right\}.
    \end{align}
    Since $A^{(0)}\sim \begin{pmatrix}
        \theta_1t&0\\
        0&\theta_2t
    \end{pmatrix}$, we find that $\dim \Ker A_1=1$.
    We also have $\dim \Ker A_2=1$ due to $\det A(a_1t)=0$.
    Therefore, applying the $q$-middle Laplace transform $\mathcal{ML}_q$ to \eqref{eq:murataDx}, we have a $q$-difference system of rank $2$ in the following form:
    \begin{align}
        D_x^-v=q\left(-\bar{S}+\frac{\bar{B}_1}{x-a_1\kappa_1/((1-q)\theta_1)}+\frac{\bar{B}_2}{x}\right)v.
    \end{align}
    {The description of $\bar{S}$ and $\bar{B}_j$ is discussed later.}
    Rewriting this equation by using $T_x^{-1}$, we have
    \begin{align}
       & T_x^{-1}v=B(x)v,
        \\
    &B(x)=I-(1-q^{-1})qx\left(-\bar{S}+\frac{\bar{B}_1}{x-a_1\kappa_1/((1-q)\theta_1)}+\frac{\bar{B}_2}{x}\right).
    \end{align}
    Therefore, we get the equation
    \begin{align}
        \label{eq:murataLaplace}
        T_xv=B(qx)^{-1}v.
    \end{align}

    In the following, we calculate $B(qx)^{-1}$.
    We first consider 
    \begin{align}
        d=\det \left[I-(1-q^{-1})qx\left(-S+\frac{B_1}{x-\tau_1}+\frac{B_2}{x}\right)\right],
    \end{align}
    where $\tau_1=a_1\kappa_1/((1-q)\theta_1)$ and
    \begin{align}
        &S=\begin{pmatrix}
            0&&&\\&0&&\\&&a_1t&\\&&&a_1t
        \end{pmatrix},
        \ 
        B_1=-\begin{pmatrix}
            1&&&\\&0&&\\&&1&\\&&&0
        \end{pmatrix}\begin{pmatrix}
            A_1&A_2\\A_1&A_2
        \end{pmatrix},\\
       &B_2=-\begin{pmatrix}
            0&&&\\&1&&\\&&0&\\&&&1
        \end{pmatrix}\begin{pmatrix}
            A_1&A_2\\A_1&A_2
        \end{pmatrix}.
    \end{align}
    By a simple linear algebra calculation, we have
    \begin{align}
        d&=\det \biggl[\begin{pmatrix}
            1&&&\\
            &1&&\\
            -1&&1&\\
            &-1&&1
        \end{pmatrix}+(1-q)x\biggl(-S-\frac{1}{x-\tau_1}\begin{pmatrix}
            1&&&\\&0&&\\&&0&\\&&&0
        \end{pmatrix}\begin{pmatrix}
            A_1&A_2\\A_1&A_2
        \end{pmatrix}
        \\
        &\qquad\qquad\qquad\qquad\qquad\qquad\qquad\qquad\qquad\qquad-\frac{1}{x}\begin{pmatrix}
            0&&&\\&1&&\\&&0&\\&&&0
        \end{pmatrix}\begin{pmatrix}
            A_1&A_2\\A_1&A_2
        \end{pmatrix}\biggr)\biggr]
        \\
        &=\det \biggl[\begin{pmatrix}
            1&&1&\\
            &1&&1\\
            -1&&&\\
            &-1&&
        \end{pmatrix}+(1-q)x\biggl(-\begin{pmatrix}
            0&&a_1t&\\
            &0&&a_1t\\
            &&0&\\
            &&&0
        \end{pmatrix}\\
        &\qquad-\frac{1}{x-\tau_1}\begin{pmatrix}
            1&&&\\&0&&\\&&0&\\&&&0
        \end{pmatrix}\begin{pmatrix}
            A_1&A_2+(1-(1-q)a_1tx)A_1\\A_1&A_2+(1-(1-q)a_1tx)A_1
        \end{pmatrix}
        \\
        &\qquad-\frac{1}{x}\begin{pmatrix}
            0&&&\\&1&&\\&&0&\\&&&0
        \end{pmatrix}\begin{pmatrix}
            A_1&A_2+(1-(1-q)a_1tx)A_1\\A_1&A_2+(1-(1-q)a_1tx)A_1
        \end{pmatrix}\biggr)\biggr]
        \\
        &=\det \biggl[
        I-(1-q)x\biggl(
            \begin{pmatrix}
                a_1t&\\&a_1t
            \end{pmatrix}
            +\frac{1}{x-\tau_1}\begin{pmatrix}
                1&0\\0&0
            \end{pmatrix}(A_2+(1-(1-q)a_1tx)A_1)
            \\
            &\qquad\qquad\qquad\qquad\qquad\qquad\qquad
            +\frac{1}{x}\begin{pmatrix}
                0&0\\0&1
            \end{pmatrix}(A_2+(1-(1-q)a_1tx)A_1)
            \biggr)
        \biggr]
        \\
        &=\frac{1}{x-\tau_1}\frac{1}{x}\det \biggl[\begin{pmatrix}
            x-\tau_1&\\&x
        \end{pmatrix}-(1-q)x\biggl(
            \begin{pmatrix}
                a_1t(x-\tau_1)&\\&a_1tx
            \end{pmatrix}
            +(A_2+(1-(1-q)a_1tx)A_1)
            \biggr)\biggr].
    \end{align}
    Since
    \begin{align}
        \begin{pmatrix}
            x-\tau_1&\\&x
        \end{pmatrix}-(1-q)x
            \begin{pmatrix}
                a_1t(x-\tau_1)&\\&a_1tx
            \end{pmatrix}=(1-(1-q)a_1tx)\begin{pmatrix}
                x-\tau_1&\\&x
            \end{pmatrix}
    \end{align}
    and
    \begin{align}
        (1-q)x(A_2+(1-(1-q)a_1tx)A_1)=(1-(1-q)a_1tx)\left(\frac{A_1}{1/((1-q)x)}+\frac{A_2}{1/((1-q)x)-a_1t}\right),
    \end{align}
    we obtain
    \begin{align}
        d&=\frac{1}{x-\tau_1}\frac{1}{x}(1-(1-q)a_1tx)^2\det\left[x I-A_\infty-\frac{A_1}{1/((1-q)x)}-\frac{A_2}{1/((1-q)x)-a_1t}\right]
        \\
        &=\frac{1}{x-\tau_1}(1-(1-q)a_1tx)^2 x \det \left[\frac{1}{\theta_1t(1-1/((1-q)a_1tx))}A\left(\frac{1}{(1-q)x}\right)\right]
        \\
        &=\frac{((1-q)a_1)^2x^3}{(x-\tau_1)\theta_1^2}\det A\left(\frac{1}{(1-q)x}\right)
        \\
        &=\frac{a_1^2}{(1-q)\theta_1^2(x-\tau_1)}\kappa_1\kappa_2\left(1-(1-q)a_1tx\right)\left(1-(1-q)a_2tx\right)\left(1-(1-q)a_3x\right).
    \end{align}
    Second, we take nonzero vectors $\begin{pmatrix}
        v_{11}\\v_{21}
    \end{pmatrix}\in \Ker A_1$, $\begin{pmatrix}
        v_{12}\\v_{22}
    \end{pmatrix}\in \Ker A_2$, and consider
    \begin{align}
        \bar{d}=\det B(x)=\det \left[I-(1-q^{-1})qx\left(-\bar{S}+\frac{\bar{B}_1}{x-\tau_1}+\frac{\bar{B}_2}{x}\right)\right],
    \end{align}
    where
    \begin{align}
        &R=\begin{pmatrix}
            1&0&v_{11}&0\\
            0&0&v_{21}&0\\
            0&1&0&v_{12}\\
            0&0&0&v_{22}
        \end{pmatrix},\\
        &R^{-1}SR=\begin{pmatrix}
            \bar{S}&\\
            \ast&S'
        \end{pmatrix},\quad
        R^{-1}B_1R=\begin{pmatrix}
            \bar{B}_1&\\
            \ast&B_1'
        \end{pmatrix},\quad
        R^{-1}B_2R=\begin{pmatrix}
            \bar{B}_2&\\
            \ast&B_2'
        \end{pmatrix}.
    \end{align}
    Due to the definition of vectors $\begin{pmatrix}
        v_{11}\\v_{21}
    \end{pmatrix}$ and $\begin{pmatrix}
        v_{12}\\v_{22}
    \end{pmatrix}$, we have $B_1'=B_2'=O$.
    We also find $\bar{S}=S'=\begin{pmatrix}
        0&\\&a_1t
    \end{pmatrix}$.
    Therefore we have 
    \begin{align}
        d=\bar{d}(1-(1-q)a_1tx).
    \end{align}
    In particular, we find
    \begin{align}
        \bar{d}=\frac{a_1^2}{(1-q)\theta_1^2(x-\tau_1)}\kappa_1\kappa_2\left(1-(1-q)a_2tx\right)\left(1-(1-q)a_3x\right).
    \end{align}
    Third, we calculate eigenvalues of $B(0)$.
    From direct calculation, we get
    \begin{align}
        \det (\lambda I-B_2)=\lambda^3(\lambda+(A_1)_{22}+(A_2)_{22}).
    \end{align}
    It follows from definitions of $A_\infty$, $A_1$, $A_2$ that
    \begin{align}
        (A_1)_{22}+(A_2)_{22}=\lim_{x\to\infty}\left(\frac{1-A(x)_{22}/(\theta_1 t(1-x/(a_1t)))}{(1-q)}\right).
    \end{align}
    We remark that $(A_{\infty})_{22}=0$.
    Since $\det A(x)=\kappa_1\kappa_2(x-a_1t)(x-a_2t)(x-a_3)$ and $A(x)_{11}=\kappa_1x^2+\cdots$, we have $A(x)_{22}=\kappa_2x+\cdots$ and
    \begin{align}
        \lim_{x\to\infty}\left(\frac{1-A(x)_{22}/(\theta_1t(1-x/(a_1t)))}{(1-q)}\right)=\frac{1+a_1\kappa_2/\theta_1}{1-q}.
    \end{align}
    Thus we find
    \begin{align}
       \det (\lambda I-(I+(1-q)B_2))=(\lambda-1)^3(\lambda+\frac{a_1\kappa_2}{\theta_1}).
    \end{align}
    On the other hand, due to $B_2\sim \begin{pmatrix}
        \bar{B}_2&\\
        &O
    \end{pmatrix}$, we have
    \begin{align}
        \det (\lambda I-(I+(1-q)B_2))=(\lambda-1)^2\det (\lambda I-(I+(1-q)\bar{B}_2)).
    \end{align}
    In conclusion, we obtain 
    \begin{align}
        \det (\lambda I-(I+(1-q)\bar{B}_2))=(\lambda-1)(\lambda+\frac{a_1\kappa_2}{\theta_1}).
    \end{align}
    We summarize the above calculations:
    \begin{itemize}
        \item $\det B(x)=\dfrac{-a_1^2}{(1-q)\tau_1\theta_1^2(1-x/\tau_1)}\kappa_1\kappa_2\left(1-(1-q)a_2tx\right)\left(1-(1-q)a_3x\right)$
        \item $B(x)=\dfrac{1}{1-x/\tau_1}(B^{(0)}+B^{(1)}x+B^{(2)}x^2)$, $B^{(0)}\sim \begin{pmatrix}
            1&\\&-\dfrac{a_1\kappa_2}{\theta_1}
        \end{pmatrix}$, $B^{(2)}=\begin{pmatrix}
            0&\\
            &\dfrac{(1-q)^2\theta_1 t}{\kappa_1}
        \end{pmatrix}$
    \end{itemize}
    Therefore we have
    \begin{itemize}
        \item $\det B(qx)^{-1}=\dfrac{-\tau_1\theta_1^2(1-qx/\tau_1)(1-q)}{a_1^2\kappa_1\kappa_2}\left(1-(1-q)a_2tqx\right)^{-1}\left(1-(1-q)a_3qx\right)^{-1}$
        \item $B(qx)^{-1}=\dfrac{1}{(1-(1-q)a_2tqx)(1-(1-q)a_3qx)}(B'^{(0)}+B'^{(1)}x+B'^{(2)}x^2)$, 
        
        $B'^{(0)}\sim \begin{pmatrix}
            1&\\&-\dfrac{\theta_1}{a_1\kappa_2}
        \end{pmatrix}$, $B'^{(2)}=\begin{pmatrix}
            -\dfrac{(1-q)^2q^2t\theta_1^2}{a_1\kappa_1\kappa_2}&\\
            &0
        \end{pmatrix}$
    \end{itemize}
    We recall the equation \eqref{eq:murataLaplace}:
    \begin{align}
        T_xv=B(qx)^{-1}v.
    \end{align}
    Putting $v'=x^{\mu}((1-q)a_2tqx)_{\infty}^{-1}((1-q)a_3qx)_{\infty}^{-1}v$ where $q^{\mu}=\theta_2t$ and $z=(1-q)qa_2a_3tx$, we finally obtain
    \begin{align}
        &T_zv'= \bar{A}(z)v',\\
        &\bar{A}(z)=\bar{A}^{(0)}+\bar{A}^{(1)}z+\bar{A}^{(2)}z^2,\\
        &\det \bar{A}(z)=\kappa_1\kappa_2(z-a_1a_2a_3\kappa_1t/\theta_1)(z-a_2t)(z-a_3),\\
        &\bar{A}^{(0)}\sim \begin{pmatrix}
            a_2a_3\kappa_1t&\\&\theta_2t
        \end{pmatrix},\quad
        \bar{A}^{(2)}=\begin{pmatrix}
            \dfrac{\theta_1}{a_2a_3}&\\&0
        \end{pmatrix}.
    \end{align}
    This completes the proof.
\end{proof}
\begin{remark}
    The transform
    \begin{align}
      s_{\mathcal{ML}_q}:a_1\mapsto -\frac{\theta_2}{\kappa_2},\ \kappa_1\mapsto \frac{\theta_1}{a_2a_3},\ \kappa_2\mapsto-\frac{\theta_2}{a_1},\ \theta_1\mapsto a_2a_3\kappa_1
    \end{align}
    is a reflection, i.e. $s_{\mathcal{ML}_q}^2=\mathrm{id}$.
    We also find braid relations:
    \begin{align}
        &s_{\mathcal{ML}_q}s_{a_1,a_2}s_{\mathcal{ML}_q}=s_{a_1,a_2}s_{\mathcal{ML}_q}s_{a_1,a_2},\\
        &s_{\mathcal{ML}_q}s_{\theta_1,\theta_2}s_{\mathcal{ML}_q}=s_{\theta_1,\theta_2}s_{\mathcal{ML}_q}s_{\theta_1,\theta_2},
    \end{align}
    where $s_{b,c}:b\leftrightarrow c$.
\end{remark}
\begin{remark}
    In \cite{Murata}, an explicit expression of the matrix $A(x)$ in  \eqref{eq:murata} is constructed:
    \begin{align}
       & A(x)=\begin{pmatrix}
            \kappa_1 ((x-y)(x-\alpha)+z_1)&w (x-y)\\
            \kappa_1 w^{-1}(\gamma x+\delta)&\kappa_2(x-y)+z_2
        \end{pmatrix},\\
        &\alpha=\frac{1}{\kappa_1}(y^{-1}((\theta_1+\theta_2) t-\kappa_1 z_1-z_2)+\kappa_2),\\
        &\gamma=z_2-\kappa_2(2y+\alpha-t(a_1+a_2)-a_3),\\
        &\delta=y^{-1}(-\kappa_2a_1a_2a_3t^2-(\alpha y+z_1)(-\kappa_2y+z_2)),
    \end{align}
    where
    \begin{align}
        z_1z_2=\kappa_2 (y-a_1t)(y-a_2t)(y-a_3).
    \end{align}
    The parameters $y$ and $z_1$ play roles of the accessory parameters, and $w$ is a gauge freedom.
    By the transformation in Theorem \ref{Thm:qPA4}, the accessory parameters $y$ and $z_1$ are transformed as follows:
    \begin{align*}
			&y\mapsto s_{\mathcal{ML}_q}(y)=-\frac{a_2 a_3 \kappa_1 y
				\left(a_2 a_3
				t-a_2 t y-a_3
				y+y^2-z_1\right)}{a_2
				a_3 \kappa_1
				z_1-a_2 a_3
				\theta_1 t+a_2
				\theta_1 t y+a_3
				\theta_1 y-\theta_1 y^2},\\
			&z_1\mapsto s_{\mathcal{ML}_q}(z_1)=\frac{a_2^2 a_3^2
				\kappa_1 z_1
				Z_1 Z_2}{\theta_1
				\left(a_2a_3\kappa_1z_1-a_2 a_3
				\theta_1 t+a_2
				\theta_1 t y+a_3
				\theta_1 y-\theta_1 y^2\right)^2},
			\\
			&Z_1=\left(-a_2^2 \kappa_1 t
			y+a_2 \kappa_1 y^2-a_2
			\kappa_1 z_1+a_2
			\theta_1 t-\theta_1 y\right),
			\\
			&Z_2=\left(-a_3^2
			\kappa_1 y+a_3 \kappa_1
			y^2-a_3 \kappa_1 z_1+a_3 \theta_1 t-\theta_1 t
			y\right).
		\end{align*}
These can be checked by a direct calculation.
We remark that the above transformation 
\begin{align}
&(a_1,a_2,a_3,\kappa_1,\kappa_2,\theta_1,\theta_2,y,z_1)\\
&\mapsto (-\frac{\theta_2}{\kappa_2},a_2,a_3,\frac{\theta_1}{a_2a_3},-\frac{\theta_2}{a_1},a_2a_3\kappa_1,\theta_2,s_{\mathcal{ML}_q}(y),s_{\mathcal{ML}_q}(z_1))
\end{align}
is a birational involution.
\end{remark}
In \cite{SasakiTakagiTakemura2022}, a symmetry of Jimbo--Sakai's Lax equation \cite{JimboSakai1996} for the $q$-Painlev\'e equation of type $D_5^{(1)}$ is described by using the $q$-middle convolution by Sakai--Yamaguchi.
In \cite{Nobukawa2026}, a rank $3$ $q$-difference equation admitting $W(E_8^{(1)})$-symmetry is obtained and a symmetry of this system is written by $mc_\lambda$.
For differential cases, to describe symmetries of equations via integral transforms (Euler-type, Laplace-type, etc.) is standard (cf. \cite{Filipuk2006,HaraokaFilipuk2007,Yamakawa2016}).
It would also be interesting and important to give a birational representation of $W(A_4^{(1)})$ from $\mathcal{ML}_q$ and compare it with the well-known symmetry of the $q$-Painlev\'e equation \cite{KNY2017}.

\section{Concluding remarks}\label{sec:summary}

\subsection{Summary}
In this paper, we constructed the $q$-middle Laplace transform $\mathcal{ML}_q$. 
We then discussed its invertibility and preservation of irreducibility (Theorem \ref{thm:qMLproperty}), a decomposition of $q$-middle convolution in terms of $\mathcal{ML}_q$ (Theorem \ref{thm:qmcinterp}), applications to the confluent $q$-hypergeometric equation (Propositions \ref{prop:qhyp1} and \ref{prop:qhyp2}) and its variant (Proposition \ref{prop:qvarhyp}), and an application to the $q$-Painlev\'e equation of type $A_4^{(1)}$ (Theorem \ref{Thm:qPA4}).

We mention a related topic.
We give a transformation of the Lax linear equation \eqref{eq:murata} for the $q$-Painlev\'e equation of type $A_4^{(1)}$ using $\mathcal{ML}_q$.
The operation $\mathcal{ML}_q$ can be applied to the Lax linear equation for the $q$-Painlev\'e equation of type $D_5^{(1)}$.
We consider Jimbo--Sakai's Lax linear equation \cite{JimboSakai1996}:
\begin{align}
    &T_xu=A(x)u,\\
    &A(x)=A(x,t)=A^{(0)}+A^{(1)}x+A^{(2)}x^2,\\
    &\det A(x)=\kappa_1\kappa_2 (x-a_1t)(x-a_2t)(x-a_3)(x-a_4),\\
    &A^{(0)}\sim \begin{pmatrix}
        \theta_1 t&0\\0&\theta_2 t
    \end{pmatrix},\ 
    A^{(2)}=\begin{pmatrix}
        \kappa_1&0\\0&\kappa_2
    \end{pmatrix}.
\end{align}
Applying the gauge transform $u\mapsto x^\rho (x/(a_1t))_\infty(x/a_3)_\infty u$ where $q^\rho=\theta_1^{-1}t^{-1}$ and $\mathcal{ML}_q$, we get an equation of the following form:
\begin{align}
    &D_x^-u=\left(B_0+\frac{B_1}{x}\right)u, \\
    &B_0=-q\begin{pmatrix}
        0&&\\
        &ta_1&\\
        &&a_3
    \end{pmatrix},
    \quad
    B_1\sim q\begin{pmatrix}
        0&&\\
        &\dfrac{a_1a_3\kappa_1-\theta_1}{(1-q)\theta_1}&\\
        &&\dfrac{a_1a_3\kappa_2-\theta_1}{(1-q)\theta_1}
    \end{pmatrix},
\end{align}
which is equivalent to Kakei--Kikuchi's Lax linear equation \cite{KakeiKikuchi}.
In particular, we can construct a transformation between Jimbo--Sakai's equation and Kakei--Kikuchi's one in terms of $\mathcal{ML}_q$.

\subsection{Future work}
\label{sec:futurework}
We outline three directions for future work.

\begin{enumerate}[(i)]

\item As suggested by Remark~\ref{rem:Harnad}, the construction of the $q$-middle Laplace transform
in this paper leads naturally to a formulation
of $q$-Harnad duality. 
In the differential setting, Harnad \cite{Harnad1994}
established the compatibility of Harnad duality with
isomonodromic deformations.
Woodhouse \cite{Woodhouse2007} extended this duality
to a broader class of isomonodromic deformation problems.
Haraoka and Filipuk \cite{HaraokaFilipuk2007}
showed that middle convolution preserves the deformation
equations of Fuchsian systems. 
For certain classes of systems with unramified irregular
singularities, Yamakawa \cite{Yamakawa2016} established
the compatibility of both Harnad duality and middle convolution
with isomonodromic deformations.

Two related problems arise in the $q$-difference setting.
One is to formulate $q$-Harnad duality in a more general setting.
The other is to establish $q$-analogs of the above results, namely, the compatibility of the
$q$-middle Laplace transform and $q$-middle convolution
with connection-preserving deformations.
We will address these problems in a subsequent paper
with Kazuki Hiroe.
That paper will also provide a detailed account of the
identification between the middle Laplace transform and the
Harnad dual functor in the differential setting,
as stated in Remark~\ref{rem:Harnad}.

\item Sakai and Yamaguchi \cite{SakaiYamaguchi}
defined spectral types and the index of rigidity for
Fuchsian $q$-difference systems and showed that
$q$-middle convolution preserves the index of rigidity
under suitable conditions. 
Although $q$-middle convolution preserves Fuchsian type,
the $q$-middle Laplace transform does not necessarily preserve it. 
Kawakami \cite{Kawakami2023} extended the notion of
spectral type to non-Fuchsian $q$-difference systems.
In the differential setting, he also studied correspondences
between spectral types induced by the Laplace transform
in connection with Harnad duality \cite{Kawakami2025}. 

These developments suggest two related problems.
One is to extend the rigidity index to non-Fuchsian
$q$-difference systems and investigate its behavior
under the $q$-middle Laplace transform.
The other is to establish $q$-analogs of the correspondences
between spectral types studied by Kawakami, using the
$q$-middle Laplace transform and, more generally,
the $q$-Harnad duality mentioned above.

\item Haraoka \cite{Haraoka2012} extended middle convolution
to linear Pfaffian systems with logarithmic singularities
along hyperplane arrangements.
The first author \cite{Adachi2026arXiv} introduced the
middle Laplace transform for linear Pfaffian systems
admitting irregular singularities.
It would be interesting to extend $q$-middle convolution
and the $q$-middle Laplace transform to systems of
$q$-difference equations in several variables and apply
them to multivariable $q$-special functions.

\end{enumerate}


\section*{Acknowledgements}
The authors thank Yasuhiko Yamada for useful discussions and advice. 
The authors also thank Kazuki Hiroe for reading the first draft and valuable comments. 
Most of this work was carried out when the second author was a member of the Department of Education in Kogakkan University.
He would like to thank the staff at Kogakkan University for thoughtful support.


\bibliographystyle{abbrv}
\bibliography{references}

\end{document}